\documentclass[11 pt]{article}
\title{\vspace{-35pt}The Johnson homomorphism and its kernel}
\author{Andrew Putman\footnote{Supported in part by NSF grant DMS-1005318}}
\date{}

\usepackage{amsmath,amssymb,amsthm,amscd,amsfonts}
\usepackage{epsfig,pinlabel}
\usepackage[hmargin=1.25in, vmargin=1in]{geometry}
\usepackage[font=small,format=plain,labelfont=bf,up,textfont=it,up]{caption}
\usepackage{stmaryrd}
\usepackage{verbatim}
\usepackage{type1cm}
\usepackage{calc}
\usepackage{microtype}
\usepackage[all,cmtip]{xy}
\usepackage{tocloft}
\usepackage{paralist}

\usepackage[bookmarks, bookmarksdepth=2, colorlinks=true, linkcolor=blue, citecolor=blue, urlcolor=blue, dvipdfmx]{hyperref}

\theoremstyle{plain}
\newtheorem{theorem}{Theorem}[section]
\newtheorem{maintheorem}{Theorem}

\newtheorem{lemma}[theorem]{Lemma}

\newtheorem{corollary}[theorem]{Corollary}

\newtheorem{claims}{Claim}
\newtheorem{stepa}{Step}
\newtheorem{stepb}{Step}

\newtheorem{substepb}{Substep}[stepb]
\newtheorem{case}{Case}

\newcommand\BeginClaims{\setcounter{claims}{0}}
\newcommand\BeginCases{\setcounter{case}{0}}

\newcommand\BeginStepsa{\setcounter{stepa}{0}}
\newcommand\BeginStepsb{\setcounter{stepb}{0}}

\newcommand\BeginSubStepsb{\setcounter{substepb}{0}}

\theoremstyle{definition}

\theoremstyle{remark}
\newtheorem*{remark}{Remark}

\DeclareMathOperator{\Ker}{ker}

\DeclareMathOperator{\Mod}{Mod}

\DeclareMathOperator{\Sp}{Sp}

\newcommand\Z{\ensuremath{\mathbb{Z}}}
\newcommand\Q{\ensuremath{\mathbb{Q}}}

\DeclareMathOperator{\HH}{H}

\DeclareMathOperator{\Interior}{Int}
\newcommand\Span[1]{\ensuremath{\langle #1 \rangle}}
\newcommand\CaptionSpace{\hspace{0.2in}}

\newcommand\Set[2]{\ensuremath{\{\text{#1 $|$ #2}\}}}

\newcommand\Figure[3]{
\begin{figure}[t]
\centering
\centerline{\psfig{file=#2,scale=55}}
\caption{#3}
\label{#1}
\end{figure}}

\newcommand\TorelliRelHat[2]{\ensuremath{\widehat{\mathcal I}(#1,#2)}}
\newcommand\TorelliRel[2]{\ensuremath{{\mathcal I}(#1,#2)}}
\newcommand\Torelli[1]{{\ensuremath{{\mathcal I}(#1)}}}
\newcommand\JKerRelHat[2]{\ensuremath{\widehat{\mathcal K}(#1,#2)}}
\newcommand\JKerRel[2]{\ensuremath{{\mathcal K}(#1,#2)}}
\newcommand\JKer[1]{\ensuremath{{\mathcal K}(#1)}}
\newcommand\Alg{\ensuremath{\text{alg}}}
\newcommand\PPM[2]{\ensuremath{\mathcal{P}(#1;#2)}}
\newcommand\PPT[3]{\ensuremath{\mathcal{P}_{\mathcal I}(#1,#2;#3)}}
\newcommand\PPK[3]{\ensuremath{\mathcal{P}_{\mathcal K}(#1,#2;#3)}}

\newcommand\PtPsh[1]{\ensuremath{\text{Push}(#1)}}
\newcommand\PtPshHat[1]{\ensuremath{\widehat{\text{Push}}(#1)}}
\newcommand{\twoheadlongrightarrow}{\relbar\joinrel\twoheadrightarrow}

\begin{document}

\maketitle

\vspace{-30pt}
\begin{abstract}
We give a new proof of a celebrated theorem of Dennis
Johnson that asserts that the kernel of the Johnson homomorphism on the Torelli subgroup
of the mapping class group is generated by separating twists.
In fact, we prove a more general result that also applies to ``subsurface Torelli groups''.  Using
this, we extend Johnson's calculation of the rational abelianization of the Torelli group
not only to the subsurface Torelli groups, but also to finite-index subgroups of the Torelli
group that contain the kernel of the Johnson homomorphism.
\end{abstract}

\tableofcontents

%%%%%%%%%%%%%%%%%%%%%%%%%%%%%%%%%%%%%%%%%%%%%%%%%%%%%%%%%%%%%%%%%%%%%%%%%%%%%%%%%%%%%%%%%%%%
%%%%%%%%%%%%%%%%%%%%%%%%%%%%%%%%%%%%%%%%%%%%%%%%%%%%%%%%%%%%%%%%%%%%%%%%%%%%%%%%%%%%%%%%%%%%
%%%%%%%%%%%%%%%%%%%%%%%%%%%%%%%%%%%%%%%%%%%%%%%%%%%%%%%%%%%%%%%%%%%%%%%%%%%%%%%%%%%%%%%%%%%%
%%%%%%%%%%%%%%%%%%%%%%%%%%%%%%%%%%%%%%%%%%%%%%%%%%%%%%%%%%%%%%%%%%%%%%%%%%%%%%%%%%%%%%%%%%%%
%%%%%%%%%%%%%%%%%%%%%%%%%%%%%%%%%%%%%%%%%%%%%%%%%%%%%%%%%%%%%%%%%%%%%%%%%%%%%%%%%%%%%%%%%%%%

\section{Introduction}

Let $\Sigma_{g}^n$ be a compact oriented genus $g$ surface with $n$ boundary components (we will often omit
the $n$ if it vanishes).
The {\em mapping class group} of $\Sigma_{g}^n$, denoted $\Mod(\Sigma_{g}^n)$, is
the group of isotopy classes of orientation preserving homeomorphisms
of $\Sigma_{g}^n$ that act as the identity on $\partial \Sigma_{g}^n$ (see \cite{FarbMargalitPrimer}).
The group $\Mod(\Sigma_{g}^n)$ acts on $\HH_1(\Sigma_{g}^n;\Z)$
and preserves the algebraic intersection form.  If $n \leq 1$,
then this is a nondegenerate alternating form,
so we obtain a representation $\Mod(\Sigma_{g}^n) \rightarrow \Sp_{2g}(\Z)$.
This representation is well-known to be surjective.
Its kernel $\Torelli{\Sigma_g^n}$ is known as the {\em Torelli group}.  This is all
summarized in the exact sequence
\[1 \longrightarrow \Torelli{\Sigma_{g}^n} \longrightarrow \Mod(\Sigma_{g}^n) \longrightarrow \Sp_{2g}(\Z) \longrightarrow 1.\]

\paragraph{Johnson's work.}
In the early 1980's, Dennis Johnson wrote a sequence of remarkable papers about $\Torelli{\Sigma_{g}^n}$ (see
\cite{JohnsonSurvey} for a survey).  Many of his results center on the
{\em Johnson homomorphisms}, which he defined in \cite{JohnsonHomo}.
Letting $H = \HH_1(\Sigma_{g}^n;\Z)$, these are homomorphisms
\[\tau_{\Sigma_{g}^1} : \Torelli{\Sigma_{g}^1} \longrightarrow \wedge^3 H \quad \text{and} \quad \tau_{\Sigma_g} : \Torelli{\Sigma_g} \longrightarrow (\wedge^3 H) / H\]
which measure the ``unipotent'' part of the action of $\Torelli{\Sigma_{g}^n}$ on the second nilpotent truncation
of $\pi_1(\Sigma_{g}^n)$.  The Johnson homomorphisms subsequently showed up in many
different parts of mathematics; for instance, $3$-manifold topology 
(see, e.g.\ \cite{GaroufalidisLevine}), the geometry of the moduli space of curves (see, e.g.\ \cite{HainTorelli, PutmanParkCity}),
and number theory (see, e.g.,\ \cite{MatsumotoParkCity}).  The kernel $\JKer{\Sigma_g^n}$ of $\tau_{\Sigma_g^n}$ is known as
the {\em Johnson kernel}.

\paragraph{Our results.}
The goal of this paper is to generalize two of Johnson's theorems.
\begin{compactenum}
\item For $g \geq 0$ and $n \leq 1$, Johnson \cite{JohnsonKg} proved that $\JKer{\Sigma_{g}^n}$ is generated by
{\em separating twists}, that is, Dehn twists $T_x$ about separating curves $x$.  
We will generalize this to the ``subsurface Torelli groups'' defined
by the author in \cite{PutmanCutPaste}.  Our proof is independent of Johnson's work, so in particular
it provides a new proof of his theorem.  Our new proof is much less computationally intensive than Johnson's
proof.
\item For $g \geq 3$ and $n \leq 1$, Johnson \cite{JohnsonAbel} proved that modulo
torsion the Johnson homomorphism gives the abelianization of $\Torelli{\Sigma_{g}^n}$.  We will generalize
this not only to subsurface Torelli groups, but also to certain finite-index subgroups of $\Torelli{\Sigma_{g}^n}$.
\end{compactenum}

\paragraph{Subsurface Torelli groups.}
At this point, we have only discussed the Torelli group on surfaces with at most one boundary
component.  Surfaces with multiple boundary components often arise during proofs by
induction on the genus of a surface.  One argues
that some phenomena is ``concentrated'' on subsurfaces of a closed surface, which are lower-genus
surfaces with boundary.  For the Torelli group, this leads to the following definition (see \cite{PutmanCutPaste}).
Let $S$ be a compact connected surface with boundary and let $S \hookrightarrow \Sigma_g^n$ be an
embedding with $n \leq 1$.  We then define $\TorelliRelHat{S}{\Sigma_g^n}$ to be the subgroup
of $\Torelli{\Sigma_g^n}$ consisting of mapping classes which are {\em supported} on $S$, that is,
which can be realized by homeomorphisms that are the identity on $\Sigma_g^n \setminus S$.  The
group so obtained depends strongly on the embedding $S \hookrightarrow \Sigma_g^n$; see
\cite{PutmanCutPaste} for a discussion.

\begin{remark}
See below for the reason for the ``hat'' in $\TorelliRelHat{S}{\Sigma_g^n}$.
\end{remark}

\paragraph{The Johnson kernel.}
For a compact connected subsurface $S$ of $\Sigma_g^n$ with $n \leq 1$, define $\JKerRelHat{S}{\Sigma_g^n}$
to be the kernel of the restriction of the Johnson homomorphism $\tau_{\Sigma_g^n}$ to $\TorelliRelHat{S}{\Sigma_g^n} \subset \Torelli{\Sigma_g^n}$.
Our first theorem is as follows.  When $S = \Sigma_g^n$, it reduces to the aforementioned theorem of Johnson giving
generators for $\JKer{\Sigma_g^n}$.

\begin{maintheorem}[Generators for Johnson kernel]
\label{maintheorem:genjohnsonker}
For $g \geq 2$ and $n \leq 1$, let $S \hookrightarrow \Sigma_g^n$ be a compact
connected subsurface of genus at least $2$.  Then $\JKerRelHat{S}{\Sigma_g^n}$
is generated by
$\{\text{$T_{\gamma}$ $|$ $\gamma$ is a separating curve on $\Sigma_g^n$ and $\gamma \subset S$}\}$.
\end{maintheorem}

\begin{remark}
We believe that the restriction on the genus of $S$ is necessary.  Let $S \hookrightarrow \Sigma_3$ be
the subsurface depicted in Figure \ref{figure:introduction} and let $x,y \subset S$ be the curves shown
there.  Then it is not hard to see that $[T_x,T_y] \in \JKerRelHat{S}{\Sigma_3}$, but we conjecture that
$[T_x,T_y]$ cannot be written as a product of separating twists lying in $S$.
\end{remark}

\begin{remark}
As indicated in the proof sketch below, the first step of our proof reduces Theorem \ref{maintheorem:genjohnsonker}
to Johnson's result.  However, the extra generality of Theorem \ref{maintheorem:genjohnsonker} is needed for the second
step, which reduces Johnson's result to understanding subsurfaces.  In other words, from our point of view
the more general statement is necessary even if you only care about the classical case.
\end{remark}

\paragraph{Abelianizing Torelli and its subgroups.}
Our second theorem is as follows.  When $S = \Sigma_g^n$, it reduces to Johnson's aforementioned theorem asserting that the Johnson homomorphism
gives the abelianization of $\Torelli{\Sigma_g^n}$ modulo torsion.

\begin{maintheorem}[Rational $\HH_1$ of finite-index subgroups of Torelli]
\label{maintheorem:abeltorelli}
For $g \geq 3$ and $n \leq 1$, let $S \hookrightarrow \Sigma_g^n$ be
a compact connected subsurface of genus at least $3$.
Let $\Gamma$ be a finite-index subgroup of $\TorelliRelHat{S}{\Sigma_g^n}$
with $\JKerRelHat{S}{\Sigma_g^n} < \Gamma$.  Then
$\HH_1(\Gamma;\Q) \cong \tau_{\Sigma_g^n}(\TorelliRelHat{S}{\Sigma_g^n}) \otimes \Q$.
\end{maintheorem}

\begin{remark}
After an early version of this paper was circulated, Dimca, Hain, and Papadima \cite{DimcaHainPapadima} calculated
$\HH_1(\JKer{\Sigma_g^n};\Q)$ for $g \geq 6$.  Their proof made use of Theorem \ref{maintheorem:abeltorelli}.
\end{remark}

\begin{remark}
Mess \cite{MessTorelli} showed that $\Torelli{\Sigma_2}$ is an infinitely generated free group.  Using this,
it is not hard to show that if $S \hookrightarrow \Sigma_g^n$ is a subsurface of genus $2$, then
$\HH_1(\TorelliRelHat{S}{\Sigma_g^n};\Q)$ has infinite rank.  In particular, the conclusion of Theorem \ref{maintheorem:abeltorelli}
does not hold.
\end{remark}

\Figure{figure:introduction}{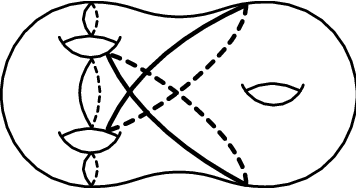}{$S \cong \Sigma_{1,3}$ is embedded in $\Sigma_3$.  We conjecture
that $[T_x,T_y]$ cannot be written as a product of separating twists in $\TorelliRelHat{S}{\Sigma_3}$}

\begin{remark}
After this paper was written, Church \cite{ChurchJohnson} showed how to define the Johnson homomorphism
on $\TorelliRelHat{S}{\Sigma_g^n}$ intrinsically, that is, without reference to the Johnson homomorphism
on $\Torelli{\Sigma_g^n}$.  This allowed him to compute the rank of $\tau_{\Sigma_g^n}(\TorelliRelHat{S}{\Sigma_g^n}) \otimes \Q$.
\end{remark}

\begin{remark}
Theorem \ref{maintheorem:abeltorelli} is inspired by a well-known conjecture of Ivanov that asserts that if
$g \geq 3$ and $\Gamma$ is a finite-index subgroup of $\Mod(\Sigma_g)$,
then $\HH_1(\Gamma;\Q) \cong \HH_1(\Mod(\Sigma_g);\Q) = 0$ (see \cite{IvanovProblems} for a recent
discussion).  In analogy with this, one might guess that the condition $\JKerRelHat{S}{\Sigma_g^n} < \Gamma$
is unnecessary.  We remark that Ivanov's conjecture is only known in a
few special cases -- Hain \cite{HainTorelli} proved it if $\Torelli{\Sigma_g} < \Gamma$, and Boggi \cite{Boggi} and
the author \cite{PutmanFiniteIndexNote} later independently proved that it holds if
$\JKer{\Sigma_g} < \Gamma$.
\end{remark}

\paragraph{About the hats.}
Let $S \hookrightarrow \Sigma_g^n$ be a subsurface with $n \leq 1$.  For simplicity, assume
that $S$ is a {\em clean} embedding, that is, no components of $\Sigma_g^n \setminus \Interior(S)$ are discs.
Instead of the subgroups $\TorelliRelHat{S}{\Sigma_g^n}$ and $\JKerRelHat{S}{\Sigma_g^n}$ of $\Mod(\Sigma_g^n)$
discussed above, most of this paper will study the groups $\TorelliRel{S}{\Sigma_g^n}, \JKerRel{S}{\Sigma_g^n} \subset \Mod(S)$
obtained by pulling these groups back under the map $\Mod(S) \rightarrow \Mod(\Sigma_g^n)$ that extends
mapping classes by the identity.  The natural surjections 
$\TorelliRel{S}{\Sigma_g^n} \rightarrow \TorelliRelHat{S}{\Sigma_g^n}$ and 
$\JKerRel{S}{\Sigma_g^n} \rightarrow \JKerRelHat{S}{\Sigma_g^n}$ are usually (but not always) isomorphisms; see
\S \ref{section:modbasic}.  Theorem \ref{maintheorem:genjohnsonker} above will derived from Theorem
\ref{theorem:genjohnsonker} below, which is an analogous (but slightly more awkward) theorem about
$\JKerRel{S}{\Sigma_g^n}$.  See the remark after Theorem \ref{theorem:birmantorellinosep} below for a discussion
of why we mostly focus on $\TorelliRel{S}{\Sigma_g^n}$ and $\JKerRel{S}{\Sigma_g^n}$.

\paragraph{Proof ideas for Theorem \ref{maintheorem:genjohnsonker}.}
Let $S$ and $\Sigma_g^n$ be as in Theorem \ref{maintheorem:genjohnsonker}.  Our proof of Theorem
\ref{maintheorem:genjohnsonker} is by induction on the genus and number of boundary components of $S$.  
The base case of the induction is when $S \cong \Sigma_2$, in which case $S$ is the entire surface $\Sigma_g^n$ (in
particular, $g=2$ and $n=0$).
It turns out that the Johnson homomorphism vanishes on $\Torelli{\Sigma_2}$, so we have $\JKer{\Sigma_2} = \Torelli{\Sigma_2}$.  We
can therefore appeal to a theorem of Powell \cite{PowellTorelli} which says that $\Torelli{\Sigma_2}$ is
generated by separating twists.  The inductive argument then has two steps.

\BeginStepsa
\begin{stepa}[Reduction to the closed case]
\label{step:closedreductionintro}
Assume that $S \cong \Sigma_h^m$ and that $\JKer{\Sigma_h}$ is generated by
separating twists.  Then $\JKerRelHat{S}{\Sigma_g^n}$ is generated by separating
twists.
\end{stepa}

This step is achieved by a version of the classical Birman exact sequence, which we will comment more
on below.  The second step of our induction is as follows.

\begin{stepa}[Reducing the genus]
\label{step:genusreductionintro}
Assume that $\JKerRelHat{S}{\Sigma_g}$ is generated by separating twists for all subsurfaces $S$ of $\Sigma_g$ satisfying $S \cong \Sigma_{g-1}^2$.
Then $\JKer{\Sigma_g}$ is generated by separating twists.
\end{stepa}

For this step, we use a result of the author \cite{PutmanSmallGenset} which says that $\Torelli{\Sigma_g}$ is
generated by elements living on certain subsurfaces of $\Torelli{\Sigma_g}$.

\begin{remark}
An earlier version of this paper performed Step \ref{step:genusreductionintro} 
by examining the action of $\JKer{\Sigma_g}$ on a simplicial
complex related to the curve complex.  Our new approach is both shorter and less technical.
\end{remark}

\paragraph{The Birman exact sequence.}
As we said above, the first step of our proof of Theorem \ref{maintheorem:genjohnsonker} uses a version of the Birman exact sequence.
The classical version of this is a fundamental tool for relating the mapping class groups of surfaces with differing numbers of boundary components.  
For $g \geq 2$ and $n \leq 1$, it takes the form
\[1 \longrightarrow \pi_1(U\Sigma_g^{n-1}) \longrightarrow \Mod(\Sigma_g^n) \longrightarrow \Mod(\Sigma_g^{n-1}) \longrightarrow 1.\]
Here $U\Sigma_g^{n-1}$ is the unit tangent bundle of $\Sigma_g^{n-1}$ and the subgroup $\pi_1(U\Sigma_g^{n-1})$ of $\Mod(\Sigma_g^n)$ is
the ``disc-pushing'' subgroup; it ``pushes'' a boundary component $\beta$ of $\Sigma_g^n$ around the surface while allowing it to rotate.
The map $\Mod(\Sigma_g^n) \rightarrow \Mod(\Sigma_g^{n-1})$ is the map obtained by gluing a disc to $\beta$ and extending mapping classes
over the disc by the identity.  A version of this for the subsurface Torelli groups was proven by the author in \cite{PutmanCutPaste}; see
\S \ref{section:torellibasic} below.  In
this paper, we give an appropriate version of this for the subsurface Johnson kernels; 
see \S \ref{section:birmanjohnsonone}--\ref{section:birmanjohnsonnosep}
for precise statements.  

\begin{remark}
After an early version of this paper was circulated, our Birman exact sequence was applied by Bestvina, Bux, and Margalit \cite{BestvinaBuxMargalit}
during their calculation of the homological dimension of the Johnson kernel.
\end{remark}

\paragraph{Proof ideas for Theorem \ref{maintheorem:abeltorelli}.}
Johnson deduced that the Johnson homomorphism gives the abelianization of the Torelli group modulo torsion by
proving that separating twists become torsion when $\Torelli{\Sigma_g^n}$ is abelianized for $g \geq 3$ and $n \leq 1$.
This reduces the result to showing that the Johnson kernel is generated by Dehn twists about
separating curves.  For Theorem \ref{maintheorem:abeltorelli}, we do the same thing; namely, we prove 
that separating twists become torsion in the abelianizations of finite-index subgroups of $\TorelliRelHat{S}{\Sigma_g^n}$
that contain $\JKerRelHat{S}{\Sigma_g^n}$.  The proof combines
a small generalization of Johnson's argument with the trick used in the author's paper \cite{PutmanFiniteIndexNote} to
show that powers of Dehn twists go to torsion in the abelianizations of finite-index subgroups of the
mapping class group.

\paragraph{Comments on Johnson's work.}
To put this paper in perspective, we now give a brief sketch of
Johnson's proof \cite{JohnsonKg} that $\JKer{\Sigma_g^n}$ is generated
by separating twists.  We will focus on the case $n=1$ and ignore some
low-genus issues.  Set $H = \HH_1(\Sigma_{g}^1;\Z)$.

Johnson's proof proceeds by a sequence of ingenious and difficult calculations.
Let $\mathcal{S} \lhd \Torelli{\Sigma_{g}^1}$ be the subgroup
generated by separating twists.  Recall that the Johnson homomorphism takes the form
\[\tau_{\Sigma_g^1} : \Torelli{\Sigma_{g}^1} \longrightarrow \wedge^3 H.\]
The free abelian group $\wedge^3 H$ has rank $\binom{g}{3}$.  Johnson begins by writing down an explicit
set $X$ of $\binom{g}{3}$ elements of $\Torelli{\Sigma_{g}^1}$ such that $\tau(X)$ is a basis for $\wedge^3 H$.  He
then does the following.
\begin{compactenum}
\item By examining the conjugates of $X$ under a standard generating set for $\Mod(\Sigma_{g}^1)$, he shows
that $X$ generates a normal subgroup of $\Mod(\Sigma_{g}^1) / \mathcal{S}$.  The set $X$ contains
a normal generating set for $\Torelli{\Sigma_{g}^1}$ identified by
Powell (\cite{PowellTorelli}; see \S \ref{section:torellibasic}
for the details of Powell's generating set).  Thus Johnson can
conclude that $X$ generates $\Torelli{\Sigma_{g}^1} / \mathcal{S}$.
\item Via a brute-force calculation, Johnson next shows that the
elements of $X$ commute modulo $\mathcal{S}$,
so $\Torelli{\Sigma_{g}^1} / \mathcal{S}$ is an abelian group generated by $\binom{g}{3}$ elements.
\item The Johnson homomorphism induces a surjection $\Torelli{\Sigma_{g}^1} / \mathcal{S} \rightarrow \wedge^3 H$.  By the
previous step, this map has to be an isomorphism, and the result follows.
\end{compactenum}

\paragraph{Outline of paper.}
In \S \ref{section:preliminaries}, we review some preliminaries about the mapping class group, the Birman
exact sequence, the Torelli group, and the Johnson homomorphisms.  We construct our version of the
Birman exact sequence for the Johnson kernel in \S \ref{section:birmanjohnson}.
We then prove Theorem \ref{maintheorem:genjohnsonker} in \S \ref{section:genjohnsonker}
and Theorem \ref{maintheorem:abeltorelli} in \S \ref{section:abeltorelli}.

\paragraph{Notation and conventions.}
In this paper, by a {\em surface} we will mean a compact oriented connected surface with boundary.  
We will denote the algebraic intersection pairing on $\HH_1(\Sigma_g;\Z)$ by $i_{\Alg}(\cdot, \cdot)$.  
A {\em symplectic basis} for $\HH_1(\Sigma_g;\Z)$ is a basis $\{a_1,b_1,\ldots,a_g,b_g\}$ for $\HH_1(\Sigma_g;\Z)$ such
that $i_{\Alg}(a_i,b_i)= 1$ and $i_{\Alg}(a_i,a_j) = i_{\Alg}(a_i,b_j) = i_{\Alg}(b_i,b_j) = 0$
for distinct $1 \leq i,j \leq g$.
If $c$ is an oriented $1$-submanifold of the surface, then $[c]$ will denote the homology class of $c$.  We will
confuse a curve on a surface with its isotopy class.  Several times we will have formulas involving $[\gamma]$, where
$\gamma$ is an unoriented simple closed curve.  By this we will mean that the formula holds for an
appropriate choice of orientation.

\paragraph{Acknowledgments.}
I would like to thank Tom Church, Benson Farb, and Dan Margalit for helpful conversations and comments.
I would also like to thank the referee for a very careful and thoughtful report.

%%%%%%%%%%%%%%%%%%%%%%%%%%%%%%%%%%%%%%%%%%%%%%%%%%%%%%%%%%%%%%%%%%%%%%%%%%%%%%%%%%%%%%%%%%%%
%%%%%%%%%%%%%%%%%%%%%%%%%%%%%%%%%%%%%%%%%%%%%%%%%%%%%%%%%%%%%%%%%%%%%%%%%%%%%%%%%%%%%%%%%%%%
%%%%%%%%%%%%%%%%%%%%%%%%%%%%%%%%%%%%%%%%%%%%%%%%%%%%%%%%%%%%%%%%%%%%%%%%%%%%%%%%%%%%%%%%%%%%
%%%%%%%%%%%%%%%%%%%%%%%%%%%%%%%%%%%%%%%%%%%%%%%%%%%%%%%%%%%%%%%%%%%%%%%%%%%%%%%%%%%%%%%%%%%%
%%%%%%%%%%%%%%%%%%%%%%%%%%%%%%%%%%%%%%%%%%%%%%%%%%%%%%%%%%%%%%%%%%%%%%%%%%%%%%%%%%%%%%%%%%%%

\section{Preliminaries}
\label{section:preliminaries}

In \S \ref{section:modbasic} and \S \ref{section:torellibasic}, we will review some
basic facts about the mapping class group and the Torelli group.  Next, in \S \ref{section:birmantorelli}
we will discuss the various versions of the Birman exact sequence for the Torelli group.
Finally, in \S \ref{section:johnsonhomomorphism}, we will discuss the Johnson homomorphism.

\subsection{Basic facts about the mapping class group}
\label{section:modbasic}

We start by discussing some basic facts about the mapping class group.

\paragraph{Subsurfaces.}
If $S$ is a surface and $S' \hookrightarrow S$ is a subsurface, then there is an induced map $i : \Mod(S') \rightarrow \Mod(S)$ (``extend
by the identity'').  This map is is injective except in the following situations.
\begin{compactitem}
\item If a component $A$ of $S \setminus \Interior(S')$ is an annulus with $\partial A \subset S'$, then letting
$x$ and $y$ be the components of $\partial A$ we have $i(T_x T_y^{-1}) = 1$.
\item If a component of $S \setminus \Interior(S')$ is a disc, then $\ker(i)$ contains the ``disc-pushing subgroup''
which forms the kernel of the Birman exact sequence (see below).
\end{compactitem}
See \cite{ParisRolfsen}, which proves that $\ker(i)$ is generated by the above elements.

\paragraph{Birman exact sequence.}
The original Birman exact sequence (see Birman \cite{BirmanExact}) dealt
with the effect on mapping class groups of punctured surfaces of ``forgetting the punctures''.  Later,
Johnson \cite{JohnsonFinite} gave an appropriate version for surfaces with boundary.  Let
$S$ be a surface of genus at least $2$ with nonempty boundary.  Let $\beta$ be a boundary
component of $S$ and let $\hat{S}$ be the result of gluing a disc to $\beta$.  The
subsurface inclusion $S \hookrightarrow \hat{S}$ induces a map $i: \Mod(S) \rightarrow \Mod(\hat{S})$.
Define $\PPM{S}{\beta} = \ker(i)$.  The group $\PPM{S}{\beta}$ is called the {\em disc-pushing subgroup}
of $\Mod(S)$.  Elements of it ``push'' the boundary component $\beta$ around $S$ while allowing it to rotate;
keeping track of this path and rotation gives an isomorphism $\PPM{S}{\beta} \stackrel{\cong}{\rightarrow} \pi_1(U\hat{S})$.  
This is all summarized in the Birman exact sequence
\begin{equation}
\label{eqn:modbirmanseq}
\begin{CD}
1  @>>> \PPM{S}{\beta} @>>> \Mod(S) @>{i}>> \Mod(\hat{S}) @>>> 1 \\
@.      @VV{\cong}V         @.              @.                 @.\\
@.      \pi_1(U\hat{S})     @.              @.                 @. \end{CD}
\end{equation}

The sequence \eqref{eqn:modbirmanseq} splits if $S$ has at least $2$ boundary components.
Say that a subsurface $S' \hookrightarrow S$
is a {\em splitting surface} for $\beta$ if $S \setminus \Interior(S')$ is a $3$-holed sphere 
two of whose boundary components are components of
$\partial S$ and 
one of whose
boundary components is $\beta$ (see Figure \ref{figure:modbirman}.a).  The composition of
inclusion maps $S' \hookrightarrow S \hookrightarrow \hat{S}$ is then homotopic to a homeomorphism and the induced map 
$\Mod(S') \rightarrow \Mod(S)$
is injective.  Regarding $\Mod(S')$ as a subgroup of $\Mod(S)$, the map $i$ restricts to an isomorphism
$\Mod(S') \cong \Mod(\hat{S})$, so \eqref{eqn:modbirmanseq} splits as
\begin{equation}
\label{eqn:modsemidirect}
\Mod(S) = \PPM{S}{\beta} \rtimes \Mod(S').
\end{equation}

\Figure{figure:modbirman}{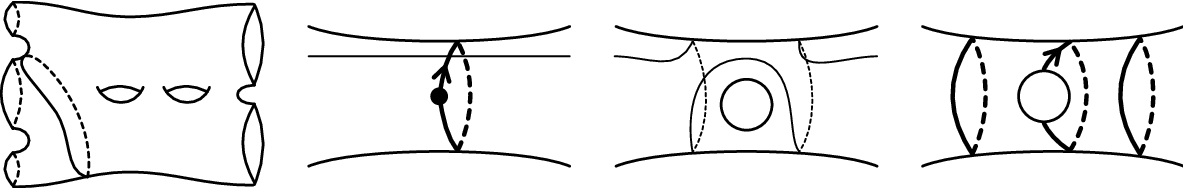}{a. The subsurface $S' \cong \Sigma_2^4$ of $S \cong \Sigma_2^5$ is
a splitting surface for $\beta$.
\CaptionSpace b. A simple closed curve $\gamma \in \pi_1(\hat{S})$.
\CaptionSpace c. The effect of $\PtPsh{\gamma}$.
\CaptionSpace d. $\PtPsh{\gamma} = T_{\tilde{\gamma}_1} T_{\tilde{\gamma}_2}^{-1}$}

\paragraph{Elements of the disc-pushing subgroup.}
We will need explicit formulas for some elements of $\PPM{S}{\beta} \cong \pi_1(U\hat{S})$.
The loop around the fiber of $U \hat{S}$ over the basepoint of $\hat{S}$ 
(with an appropriate orientation) corresponds to $T_{\beta}$.  
It is easiest to understand the other elements via the projection
$\PPM{S}{\beta} \cong \pi_1(U\hat{S}) \rightarrow \pi_1(\hat{S})$.
Consider $\gamma \in \pi_1(\hat{S})$.  Assume that 
$\gamma$ can be realized by a simple closed curve, which we can
assume is smoothly embedded in $\hat{S}$.  
Taking the derivative of such an embedding and rescaling to get a unit vector, we get a well-defined lift
$\PtPsh{\gamma} \in \pi_1(U\hat{S}) \cong \PPM{S}{\beta} < \Mod(S)$.

\begin{remark}
This is well-defined since any two realizations of $\gamma$ as a smoothly embedded simple closed curve are smoothly isotopic
through an isotopy fixing the basepoint.
\end{remark}

As shown in Figure \ref{figure:modbirman}.b--d, for such a $\gamma$ we have
$\PtPsh{\gamma} = T_{\tilde{\gamma}_1} T_{\tilde{\gamma}_2}^{-1}$ for simple closed curves $\tilde{\gamma}_1$ and
$\tilde{\gamma}_2$ in $S$ constructed as follows.  Let $\rho : S \rightarrow \hat{S}$
be the map which collapses $\beta$ to a point (the basepoint of $\hat{S}$).  Then
$\rho^{-1}(\gamma)$ consists of $\beta$ together with an oriented arc joining two points of $\beta$.  The boundary
of a regular neighborhood of $\rho^{-1}(\gamma)$ has two components, and the curves $\tilde{\gamma}_1$ and
$\tilde{\gamma}_2$ are the components lying to the right and left of the arc, respectively.

\subsection{Basic facts about the Torelli group}
\label{section:torellibasic}

We now turn to the Torelli group.  Recall that for $g \geq 0$ and $n \leq 1$ this is the 
kernel $\Torelli{\Sigma_g^n}$ of the action of $\Mod(\Sigma_g^n)$ on $\HH_1(\Sigma_g^n;\Z)$.

\paragraph{Generators.}
Following work of Birman \cite{BirmanSiegel}, Powell \cite{PowellTorelli} proved that
$\Torelli{\Sigma_g^n}$ is generated by the following elements for $g \geq 0$ and $n \leq 1$.
\begin{compactitem}
\item Separating twists, that is, Dehn twists $T_{x}$ about separating simple closed curves $x$
(see Figure \ref{figure:torellibirman}.a).
\item Bounding pair maps, that is, products $T_y T_z^{-1}$ for 
disjoint nonseparating homologous simple closed curves $y$ and $z$,
(so called because $y \cup z$ separates $\Sigma_g^n$; see Figure
\ref{figure:torellibirman}.a).
\end{compactitem}

\paragraph{Subsurface Torelli groups.}
Thus far, we have only defined the Torelli group on a surface $\Sigma_g^n$ with $n \leq 1$.
Now consider an arbitrary surface $S$.  Fix an embedding $S \hookrightarrow \Sigma_g^n$ with $n \leq 1$,
let $i : \Mod(S) \rightarrow \Mod(\Sigma_g^n)$ be the induced map, and define
\[\TorelliRel{S}{\Sigma_g^n} = i^{-1}(\Torelli{\Sigma_g^n}).\]
The subgroup $\TorelliRel{S}{\Sigma_g^n}$ of $\Mod(S)$ so obtained depends
on the embedding $S \hookrightarrow \Sigma_g^n$.  In \cite{PutmanCutPaste},
the author showed that the Torelli groups on $S$ associated to different
embeddings into a surface with at most one boundary component are parametrized
by a partition of the boundary components of $S$, but we will not need this.

\Figure{figure:torellibirman}{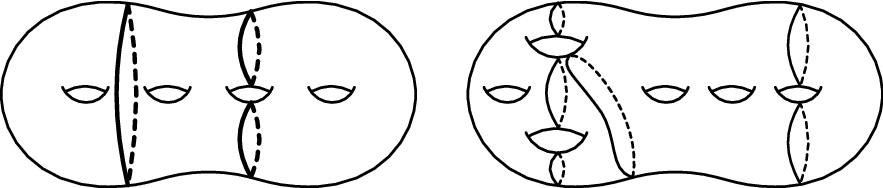}{a. A separating twist $T_x$ and a bounding pair
map $T_y T_z^{-1}$.
\CaptionSpace b. The subsurface $S \cong \Sigma_2^5$ of $\Sigma_6$ is bounded by the five ``vertical'' curves.  The subsurface
$S' \cong \Sigma_2^4$ of $S$ is a $\Sigma_6$-splitting surface for $\beta$.}

\begin{remark}
Defining $\TorelliRelHat{S}{\Sigma_g^n}$ as in the introduction, 
we have $i(\TorelliRel{S}{\Sigma_g^n}) = \TorelliRelHat{S}{\Sigma_g^n}$.  Also, $i$ restricts to an isomorphism
from $\TorelliRel{S}{\Sigma_g^n}$ to $\TorelliRelHat{S}{\Sigma_g^n}$ except for the situations discussed
in \S \ref{section:modbasic} where $i$ is not injective.
\end{remark}

\begin{lemma}
\label{lemma:degeneratetorelli}
The following hold for all $n \leq 1$.
\begin{compactenum}[(a)]
\item If $S \cong \Sigma_h^1$ and $S \hookrightarrow \Sigma_g^n$ is any embedding, then
$\TorelliRel{S}{\Sigma_g^n} = \Torelli{S} \cong \Torelli{\Sigma_h^1}$.
\item If $S \hookrightarrow \Sigma_g^1$ and $\Sigma_g^1 \hookrightarrow \Sigma_h^n$ are
embeddings, then $\TorelliRel{S}{\Sigma_g^1} = \TorelliRel{S}{\Sigma_h^n}$.  
\end{compactenum}
\end{lemma}
\begin{proof}
Trivial (see \cite[Theorem Summary 1.1]{PutmanCutPaste} for a more general result).
\end{proof}

\noindent
The paper \cite{PutmanCutPaste} extends Powell's computation of generators
of $\Torelli{\Sigma_g^n}$ to the groups $\TorelliRel{S}{\Sigma_g^n}$.  If $x$
is a nonnullhomotopic simple closed curve on $S$, then say that
$T_x \in \Mod(S)$ is a {\em $\Sigma_g^n$-separating twist} if $x$ separates
$\Sigma_g^n$.  Similarly, if $y$ and $z$ are disjoint nonnullhomotopic simple closed curves
on $S$, then say that $T_y T_z^{-1} \in \Mod(S)$ is a {\em $\Sigma_g^n$-bounding pair map}
if $y$ and $z$ are nonseparating and homologous in $\Sigma_g^n$.  Then
$\TorelliRel{S}{\Sigma_g^n}$ is generated by the set of $\Sigma_g^n$-separating twists
and $\Sigma_g^n$-bounding pair maps if the genus of $S$ is at least $1$ (see \cite[Theorem 1.3]{PutmanCutPaste}).

\subsection{Birman exact sequence for Torelli group}
\label{section:birmantorelli}

Let $S \hookrightarrow \Sigma_g^n$ be a subsurface with $n \leq 1$ and let $\beta$ be a boundary
component of $S$.  Define $\PPT{S}{\Sigma_g^n}{\beta} = \PPM{S}{\beta} \cap \TorelliRel{S}{\Sigma_g^n}$.
The Birman exact sequence for the Torellli group (proved by the author
in \cite{PutmanCutPaste}) relates $\TorelliRel{S}{\Sigma_g^n}$ to a Torelli group
defined on the result of gluing a disc to $\beta$.  There are three versions.

\paragraph{Birman exact sequence for Torelli, one boundary component.}
The first is as follows,
which mildly generalizes a theorem of Johnson \cite{JohnsonFinite}.

\begin{theorem}[{\cite[Theorem 1.2]{PutmanCutPaste}}]
\label{theorem:birmantorellione}
Let $S \hookrightarrow \Sigma_g^n$ be a subsurface with $n \leq 1$ and let $\beta$ be a boundary
component of $S$.  Assume that the genus of $S$ is at least $2$ and that $\beta$ is the only
boundary component of $S$. Let $\hat{S}$ be the result of gluing a disc to $\beta$.  Then
there is an exact sequence
\[1 \longrightarrow \PPT{S}{\Sigma_g^n}{\beta} \longrightarrow \TorelliRel{S}{\Sigma_g^n} \longrightarrow \Torelli{\hat{S}} \longrightarrow 1.\]
Also, $\PPT{S}{\Sigma_g^n}{\beta} = \PPM{S}{\beta} \cong \pi_1(U\hat{S})$.
\end{theorem}

\begin{remark}
Observe that in Theorem \ref{theorem:birmantorellione}, the surface $\hat{S}$ is closed, so $\Torelli{\hat{S}}$ makes sense.
\end{remark}

\paragraph{Birman exact sequence for Torelli, separating boundary.}
The second version of the Birman exact sequence for the Torelli group is as follows.  In it
(and throughout this paper), we regard a boundary component of a surface $\Sigma_g^n$ as being
a separating curve on $\Sigma_g^n$.

\begin{theorem}[{\cite[Theorem 1.2]{PutmanCutPaste}}]
\label{theorem:birmantorellisep}
Let $S \hookrightarrow \Sigma_g^n$ be a subsurface with $n \leq 1$ and let $\beta$ be a boundary
component of $S$.  Assume that the genus of $S$ is at least $2$, that $S$ has more than
one boundary component, and that $\beta$ is a separating curve on $\Sigma_g^n$.
Let $S' \hookrightarrow S$ be a splitting surface for $\beta$ and let $\hat{S}$ be
the result of gluing a disc to $\beta$.  Then
\[\TorelliRel{S}{\Sigma_g^n} = \PPT{S}{\Sigma_g^n}{\beta} \rtimes \TorelliRel{S'}{\Sigma_g^n}.\]
Also, $\PPT{S}{\Sigma_g^n}{\beta} = \PPM{S}{\beta} \cong \pi_1(U\hat{S})$.
\end{theorem}

\paragraph{Subsurface homology map.}
To deal with the case where $\beta$ does not separate $\Sigma_g^n$, we need to introduce
the {\em subsurface homology map} of $S$.  We will actually need this map for separating
boundary components too, so $\beta$ is still an arbitrary boundary component of $S$.  Let
$\hat{S}$ be the surface obtained by gluing a disc to $S$ along $\beta$.  
The inclusion map $S \hookrightarrow \Sigma_g^n$ induces a map
$\HH_1(S;\Z) \rightarrow \HH_1(\Sigma_g^n;\Z)$.  This in turn descends to a map
$\HH_1(\hat{S};\Z) \rightarrow \HH_1(\Sigma_g^n;\Z) / \Span{[\beta]}$ which we will
call the {\em subsurface homology map} of $S$ and will denote $\sigma_{S,\Sigma_g^n,\beta}$ 

\begin{remark}
If $\beta$ separates $\Sigma_g^n$,
then $\sigma_{S,\Sigma_g^n,\beta}$ has codomain $\HH_1(\Sigma_g^n;\Z)$.
\end{remark}

\paragraph{Relative commutator subgroup.}
We will denote by $[\pi_1(\hat{S}),\pi_1(\hat{S})]_{\Sigma_g^n}$ the kernel of the composition
\[\pi_1(\hat{S}) \rightarrow \HH_1(\hat{S};\Z) \stackrel{\sigma_{S,\Sigma_g^n,\beta}}{\longrightarrow} \HH_1(\Sigma_g^n;\Z) / \Span{[\beta]}.\]
If $K \subset \HH_1(\hat{S};\Z)$ is the kernel of $\sigma_{S,\Sigma_g^n,\beta}$, then we have a short
exact sequence
\begin{equation}
\label{eqn:decomposecommutator}
1 \longrightarrow [\pi_1(\hat{S}),\pi_1(\hat{S})] \longrightarrow [\pi_1(\hat{S}),\pi_1(\hat{S})]_{\Sigma_g^n} \longrightarrow K \longrightarrow 1.
\end{equation}
We will need the following observation.

\begin{lemma}
\label{lemma:relativecommutatorunchanged}
Let $S \hookrightarrow \Sigma_g^1$ and $\Sigma_g^1 \hookrightarrow \Sigma_h^n$ be
subsurfaces with $n \leq 1$.  Let $\beta$ be a boundary component of $S$
and let $\hat{S}$ be the result of gluing a disc to $\beta$.  Then
$[\pi_1(\hat{S}),\pi_1(\hat{S})]_{\Sigma_g^1} = [\pi_1(\hat{S}),\pi_1(\hat{S})]_{\Sigma_h^n}$.
\end{lemma}
\begin{proof}
Since the map $\HH_1(\Sigma_g^1;\Z) \rightarrow \HH_1(\Sigma_h^n;\Z)$ is injective,
the kernels of $\sigma_{S,\Sigma_g^1,\beta}$ and $\sigma_{S,\Sigma_h^n,\beta}$ are the
same.  The lemma now follows from \eqref{eqn:decomposecommutator}.
\end{proof}

\paragraph{$\mathbf{\Sigma_g^n}$-splitting surfaces.}
Say that a splitting surface $S' \hookrightarrow S$ for $\beta$
is a {\em $\Sigma_g^n$-splitting surface} if $\Sigma_g^n \setminus \Interior(S')$ has
the same number of components as $\Sigma_g^n \setminus \Interior(S)$ (see Figure
\ref{figure:torellibirman}.b; in the language of \cite{PutmanCutPaste}, this assumption implies
that both boundary components of $S \setminus S'$ lie in the same partition element).
Observe that a $\Sigma_g^n$-splitting surface exists for $\beta$ if and only if
$\beta$ is a nonseparating curve on $\Sigma_g^n$.

\paragraph{Birman exact sequence for Torelli, nonseparating boundary.}
The third version of the Birman exact sequence for the Torelli group is as follows. 

\begin{theorem}[{\cite[Theorem 1.2]{PutmanCutPaste}}]
\label{theorem:birmantorellinosep}
Let $S \hookrightarrow \Sigma_g^n$ be a subsurface with $n \leq 1$ and let $\beta$ be a boundary
component of $S$.  Assume that the genus of $S$ is at least $2$ and 
that $\beta$ is a nonseparating curve on $\Sigma_g^n$.
Let $S' \hookrightarrow S$ be a $\Sigma_g^n$-splitting surface for $\beta$ and let $\hat{S}$ be
the result of gluing a disc to $\beta$.  Then
\[\TorelliRel{S}{\Sigma_g^n} = \PPT{S}{\Sigma_g^n}{\beta} \rtimes \TorelliRel{S'}{\Sigma_g^n}.\]
Also, $\PPT{S}{\Sigma_g^n}{\beta} \cong [\pi_1(\hat{S}),\pi_1(\hat{S})]_{\Sigma_g^n}$.
\end{theorem}

\begin{remark}
Unlike $\TorelliRel{S}{\Sigma_g^n} \subset \Mod(S)$ in Theorem \ref{theorem:birmantorellinosep}, the group
$\TorelliRelHat{S}{\Sigma_g^n} \subset \Mod(\Sigma_g^n)$ does {\em not} always split as a semidirect product
(there is a Birman exact sequence, but it need not be split; this issue arises when the component
$C$ of $\Sigma_g^n \setminus \Interior(S)$ containing $\beta$ is an annulus).  This is the whole reason
we went to the trouble of introducing the groups $\TorelliRel{S}{\Sigma_g^n}$ -- the above semidirect
product decomposition will be very convenient, especially when we introduce the Johnson homomorphism.
\end{remark}

\paragraph{The embedding.}
For the rest of this section, let the notation be as in Theorem \ref{theorem:birmantorellinosep}.
The assumptions about $\beta$ in that
theorem imply that $\hat{S}$ has a nonempty boundary,
so there is an (unnatural) direct sum decomposition $\pi_1(U\hat{S}) \cong \pi_1(\hat{S}) \times \Z$.  However, though
$\PPT{S}{\Sigma_g^n}{\beta}$ is isomorphic to a subgroup of $\pi_1(\hat{S})$, it does {\em not} lie in the first
factor of such a direct sum decomposition.  Rather, for every such decomposition there exists a homomorphism $\psi : [\pi_1(\hat{S}),\pi_1(\hat{S})]_{\Sigma_g^n} \rightarrow \Z$
such that
\begin{equation}
\label{eqn:embedcommutator}
\PPT{S}{\Sigma_g^n}{\beta} \cong \Set{$(x,\psi(x))$}{$x \in [\pi_1(\hat{S}),\pi_1(\hat{S})]_{\Sigma_g^n}$} \subset \pi_1(\hat{S}) \times \Z \cong \pi_1(U\hat{S}).
\end{equation}
For a proof of this, see \cite[Theorem 4.1]{PutmanCutPaste}.

\paragraph{Generating the relative commutator subgroup.}
We will need to know generators for the group $[\pi_1(\hat{S}),\pi_1(\hat{S})]_{\Sigma_g^n}$ that
appears in Theorem \ref{theorem:birmantorellinosep}.
The first type of generator are {\em torus-bounding curves}, which are 
$\gamma \in [\pi_1(\hat{S}),\pi_1(\hat{S})]$ that can be realized by
simple closed curves that separate $S$ into two subsurfaces, one of which is a one-holed torus
(see Figures \ref{figure:explicitformulas}.a,b).  The following lemma says that these generate
$[\pi_1(\hat{S}),\pi_1(\hat{S})]$.

\begin{lemma}[{\cite[Theorem A.1]{PutmanCutPaste}}]
\label{lemma:commgen}
Let $X$ be a surface whose genus is at least $1$.  Pick $\ast \in \Interior(X)$.  Then $[\pi_1(X,\ast),\pi_1(X,\ast)]$
is generated by the set of elements that can be realized by torus-bounding curves.
\end{lemma}

\noindent
The other type of element needed to generate $[\pi_1(\hat{S}),\pi_1(\hat{S})]_{\Sigma_g^n}$ are
{\em $\Sigma_g^n$-boundary curves}, which are elements $\gamma \in \pi_1(\hat{S})$ that can be
realized by simple closed curve such that there exists a component $M$ of
$\Sigma_g^n \setminus \Interior(S)$ satisfying the following two conditions (see Figures \ref{figure:explicitformulas}.a,d).
\begin{compactitem}
\item $M \cap \beta = \emptyset$, so $M \cap S$ lies in $\hat{S}$, and
\item $\gamma$ separates $\hat{S}$ into two components,
one of which is homeomorphic to a sphere whose boundary components consist of $\gamma$ together
with $M \cap S \subset \hat{S}$.
\end{compactitem}
We will say that $\gamma$ {\em surrounds} $M$.  We then have the following lemma.

\begin{lemma}
\label{lemma:relcommgen}
Let $S$ be a surface whose genus is at least $1$ and let $S \hookrightarrow \Sigma_g^n$ be an embedding
with $n \leq 1$.  Let $\beta$ be a boundary component of $S$ and let $\hat{S}$ be the result of gluing a disc
to $S$ along $\beta$.  Then $[\pi_1(\hat{S}),\pi_1(\hat{S})]_{\Sigma_g^n}$ is generated by the set of elements
that can be realized by torus-bounding curves or $\Sigma_g^n$-boundary curves.
\end{lemma}
\begin{proof}
Let $\Gamma$ be the subgroup of $[\pi_1(\hat{S}),\pi_1(\hat{S})]_{\Sigma_g^n}$ generated by the set
of elements that can be realized by torus-bounding curves or $\Sigma_g^n$-boundary curves.  By Lemma
\ref{lemma:commgen}, the subgroup $\Gamma$ contains the entire kernel of the exact sequence
\eqref{eqn:decomposecommutator}.  Also, the set of elements that can be realized by $\Sigma_g^n$-boundary
curves project to a generating set for the cokernel of \eqref{eqn:decomposecommutator}.  We conclude
that $\Gamma = [\pi_1(\hat{S}),\pi_1(\hat{S})]_{\Sigma_g^n}$, as desired.
\end{proof}

\paragraph{Explicit formulas.}
For $\gamma \in [\pi_1(\hat{S}),\pi_1(\hat{S})]_{\Sigma_g^n}$, denote by $\PtPshHat{\gamma} \in \PPT{S}{\Sigma_g^n}{\beta}$ the associated
element of Torelli.  We will need some explicit formulas for $\PtPshHat{\gamma}$.  These can be obtained as follows.  Assume that
$\gamma \in [\pi_1(\hat{S}),\pi_1(\hat{S})]_{\Sigma_g^n}$ can be realized by a simple closed curve.  We therefore have
the element $\PtPsh{\gamma} \in \PPM{S}{\beta}$, which projects to $\gamma$ under the natural projection 
$\PPM{S}{\beta} \cong \pi_1(U\hat{S}) \rightarrow \pi_1(\hat{S})$.  Examining \eqref{eqn:embedcommutator}, we see that there is a unique
$k \in \Z$ such that $\PtPsh{\gamma} T_{\beta}^k \in \PPT{S}{\Sigma_g^n}{\beta}$.  We then have $\PtPshHat{\gamma} = \PtPsh{\gamma} T_{\beta}^k$.  Using this recipe, one can show that if $\gamma$ can be realized by a torus-bounding curve
or a $\Sigma_g^n$-boundary curve, then $\PtPshHat{\gamma}$ is the product of a $\Sigma_g^n$-separating twist and a 
$\Sigma_g^n$-bounding pair.  For example, consider the subsurface $S$ of $\Sigma_4$ in Figure \ref{figure:explicitformulas}.a.
Let $\gamma$ be the torus-bounding curve depicted in Figure
\ref{figure:explicitformulas}.b.  We then have $\PtPsh{\gamma} = T_{\tilde{\gamma}_1} T_{\tilde{\gamma}_2}^{-1}$
for $\tilde{\gamma}_1$ and $\tilde{\gamma}_2$ as in Figure \ref{figure:explicitformulas}.c.
The mapping class $T_{\tilde{\gamma}_1}$ is a $\Sigma_4$-separating
twist, so $T_{\tilde{\gamma}_1} \in \TorelliRel{S}{\Sigma_4}$.  The mapping class $T_{\tilde{\gamma}_2}$ is {\em not} a $\Sigma_4$-separating twist, but
$T_{\tilde{\gamma}_2} T_{\beta}^{-1}$ is a $\Sigma_4$-bounding pair map, so $T_{\tilde{\gamma}_2} T_{\beta}^{-1} \in \TorelliRel{S}{\Sigma_4}$.
We conclude that $\PtPshHat{\gamma} = T_{\tilde{\gamma}_1} (T_{\beta} T_{\tilde{\gamma}_2}^{-1})$

\Figure{figure:explicitformulas}{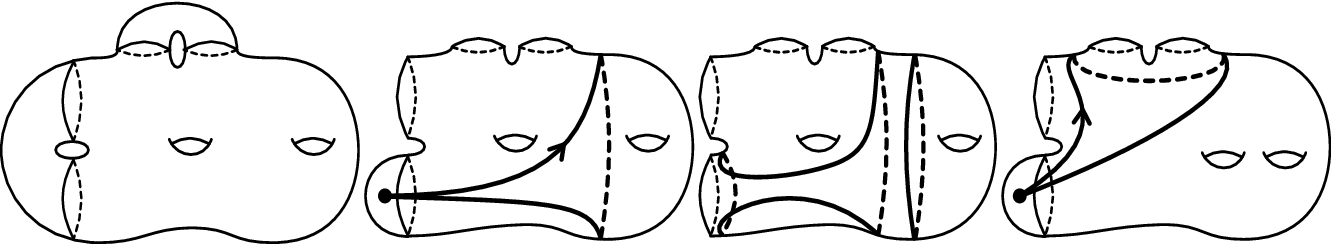}{a. A subsurface $S \cong \Sigma_2^4$ of $\Sigma_4$ and a boundary component $\beta$ of $S$ \CaptionSpace
b. A torus-bounding curve $\gamma \in \pi_1(\hat{S})$ \CaptionSpace
c. We have $\PtPsh{\gamma} = T_{\tilde{\gamma}_1} T_{\tilde{\gamma}_2}^{-1}$ and
$\PtPshHat{\gamma} = T_{\tilde{\gamma}_1} (T_{\beta} T_{\tilde{\gamma}_2}^{-1})$. \CaptionSpace
d. A $\Sigma_4$-bounding curve $\gamma \in \pi_1(\hat{S})$.}

\subsection{The Johnson homomorphism and kernel}
\label{section:johnsonhomomorphism}

The Johnson homomorphism was originally defined in \cite{JohnsonHomo} by examining
the action of the Torelli group on nilpotent truncations of the surface group.  There
are now several alternate definitions (see \cite{JohnsonSurvey}).
We will only need its formal properties.

\paragraph{Characteristic element.}
Let $V \cong \Z^{2g}$ be a free $\Z$-module equipped with a symplectic form.  The {\em characteristic element} 
$\omega_V \in \wedge^2 V$ is $a_1 \wedge b_1 + \cdots + a_g \wedge b_g$, where $\{a_1,b_1,\ldots,a_g,b_g\}$ is
a symplectic basis for $V$.  The element $\omega_V$ is independent of the choice of symplectic basis.

\paragraph{Surfaces with one boundary component.}
Set $H = \HH_1(\Sigma_{g}^1;\Z)$.  The Johnson homomorphism on $\Torelli{\Sigma_g^1}$ is a 
surjective homomorphism $\tau_{\Sigma_g^1} : \Torelli{\Sigma_g^1} \rightarrow \wedge^3 H$.  It
has the following behavior on our generators for $\Torelli{\Sigma_g^1}$.
\begin{compactitem}
\item If $T_x \in \Torelli{\Sigma_g^1}$ is a separating twist, then $\tau_{\Sigma_g^1}(T_x) = 0$.  
\item If $T_y T_z^{-1} \in \Torelli{\Sigma_g^1}$ is a bounding pair map, then
$\tau_{\Sigma_g^1}(T_y T_z^{-1}) = [y] \wedge \omega_V$ where $[y] \in H$ and $\omega_V \in \wedge^2 H$ 
are as follows.
The curves $y \cup z$ divide $\Sigma_g^1$ into two subsurfaces $Q_1$ and $Q_2$.  Order
them so that $\partial \Sigma_g^1 \subset Q_2$.  We then have $Q_1 \cong \Sigma_h^2$ and
$Q_2 \cong \Sigma_{g-h-1}^3$ for some $0<h<g$.  Orienting $y$ so that $Q_1$ lies on its
right, the element $[y] \in \HH_1(\Sigma_g^1;\Z)$ is the homology class of $y$.
Let $Q_1' \subset Q_1$ be a genus $h$ subsurface
with $1$ boundary component.  Thus $V:=\HH_1(Q_1';\Z)$ is a symplectic $\Z$-module with
a characteristic element $\omega_V \in \wedge^2 V$, which we will identify with
an element of $\wedge^2 H$ via the inclusion $\wedge^2 V \subset \wedge^2 H$.  
\end{compactitem}

\begin{remark}
It is instructive to verify that $[y] \wedge \omega_V$ is independent of the choice of $Q_1'$.
\end{remark}

The Johnson kernel $\JKer{\Sigma_g^1}$ is the kernel of $\tau_{\Sigma_g^1}$.

\paragraph{Closed surfaces.}
Continue to let $H = \HH_1(\Sigma_g^1;\Z)$.  Let $\beta$ be the boundary component of
$\Sigma_g^1$ and let $\omega \in \wedge^2 H$ be the characteristic element.  Johnson \cite{JohnsonAbel}
proved the following lemma.  In it, we identify $H = \HH_1(\Sigma_g^1;\Z)$ with $\HH_1(\Sigma_g;\Z)$ in the
obvious way.  Also, we use the inclusion $\PPM{\Sigma_g^1}{\beta} \subset \Torelli{\Sigma_g^1}$ from
Theorem \ref{theorem:birmantorellione}.

\begin{lemma}
\label{lemma:ptpushjohnsonclassical}
The restriction of $\tau_{\Sigma_g^1}$ to the disc-pushing subgroup $\PPM{\Sigma_g^1}{\beta} \cong \pi_1(U\Sigma_g)$
is the composition
\[\pi_1(U\Sigma_g) \longrightarrow \pi_1(\Sigma_g) \longrightarrow H \stackrel{\kappa}{\longrightarrow} \wedge^3 H,\]
where $\kappa : H \rightarrow \wedge^3 H$ is the inclusion $\kappa(h) = h \wedge \omega$.  
\end{lemma}

There is therefore an induced
Johnson homomorphism $\tau_{\Sigma_g} : \Torelli{\Sigma_g} \rightarrow (\wedge^3 H) / H$ which fits
into a commutative diagram of short exact sequences
\begin{equation}
\label{eqn:johnsonhomodiagram}
\begin{CD}
1 @>>> \PPM{\Sigma_g^1}{\beta} @>>>    \Torelli{\Sigma_g^1}    @>>> \Torelli{\Sigma_g}    @>>> 1 \\
@.     @VVV                               @VV{\tau_{\Sigma_g^1}}V      @VV{\tau_{\Sigma_g}}V      @.\\
1 @>>> H                          @>{\kappa}>> \wedge^3 H              @>>> (\wedge^3 H)/H        @>>> 1.
\end{CD}
\end{equation}
The Johnson kernel $\JKer{\Sigma_g}$ is the kernel of $\tau_{\Sigma_g}$.

%%%%%%%%%%%%%%%%%%%%%%%%%%%%%%%%%%%%%%%%%%%%%%%%%%%%%%%%%%%%%%%%%%%%%%%%%%%%%%%%%%%%%%%%%%%%
%%%%%%%%%%%%%%%%%%%%%%%%%%%%%%%%%%%%%%%%%%%%%%%%%%%%%%%%%%%%%%%%%%%%%%%%%%%%%%%%%%%%%%%%%%%%
%%%%%%%%%%%%%%%%%%%%%%%%%%%%%%%%%%%%%%%%%%%%%%%%%%%%%%%%%%%%%%%%%%%%%%%%%%%%%%%%%%%%%%%%%%%%
%%%%%%%%%%%%%%%%%%%%%%%%%%%%%%%%%%%%%%%%%%%%%%%%%%%%%%%%%%%%%%%%%%%%%%%%%%%%%%%%%%%%%%%%%%%%
%%%%%%%%%%%%%%%%%%%%%%%%%%%%%%%%%%%%%%%%%%%%%%%%%%%%%%%%%%%%%%%%%%%%%%%%%%%%%%%%%%%%%%%%%%%%

\section{The subsurface Johnson homomorphisms and kernels}
\label{section:subsurfacejohnson}

Let $S$ be an arbitrary surface and let $S \hookrightarrow \Sigma_g^n$ be an embedding with
$n \leq 1$.  We want to define the Johnson homomorphism on $\TorelliRel{S}{\Sigma_g^n}$.  To
avoid unnecessary technicalities, we will restrict ourselves to {\em clean embeddings}
$S \hookrightarrow \Sigma_g^n$, that is, embeddings such that no components of
$\Sigma_g^n \setminus \Interior(S)$ are discs.  Given a clean embedding $S \hookrightarrow \Sigma_g^n$ with $n \leq 1$,
we define the Johnson homomorphism $\tau_{S,\Sigma_g^n}$ on $\TorelliRel{S}{\Sigma_g^n}$ as follows.
Let $H = \HH_1(\Sigma_g^n;\Z)$.
\begin{compactitem}
\item If $n=1$, then $\tau_{S,\Sigma_g^n}$ is the composition
\[\TorelliRel{S}{\Sigma_g^n} \longrightarrow \Torelli{\Sigma_g^n} \stackrel{\tau_{\Sigma_g^n}}{\longrightarrow} \wedge^3 H.\]
\item If $n=0$, then $\tau_{S,\Sigma_g^n}$ is the composition 
\[\TorelliRel{S}{\Sigma_g^n} \longrightarrow \Torelli{\Sigma_g^n} \stackrel{\tau_{\Sigma_g^n}}{\longrightarrow} (\wedge^3 H)/H.\]
\end{compactitem}
We emphasize that $\tau_{S,\Sigma_g^n}$ is not defined if $S \hookrightarrow \Sigma_g^n$ is not clean (we do this
to avoid complicated and unnecessary special cases in our results).  
We define the Johnson kernel $\JKerRel{S}{\Sigma_g^n}$ to be the kernel of $\tau_{S,\Sigma_g^n}$.

To understand the relationships between different subsurface Johnson homomorphisms, we will
need the following lemma.  In it, the orthogonal complement
is taken with respect to the algebraic intersection pairing.

\begin{lemma}
\label{lemma:linearalgebra}
For some $g \geq 1$, set $H = \HH_1(\Sigma_g)$ and let $\omega \in \wedge^2 H$ be the characteristic element.  Define
$\kappa : H \rightarrow \wedge^3 H$ by $\kappa(h) = h \wedge \omega$.
\begin{compactenum}[(a)]
\item If $x$ is a primitive vector in $H$, then $(\wedge^3 (x^{\perp})) \cap \kappa(H) = \Span{\kappa(x)} \cong \Z$.
\item If $x$ and $y$ are primitive vectors in $H$ with $i_{\Alg}(x,y)$ equal to $0$ or $1$
that span a rank-$2$ summand of $H$, then $(\wedge^3 (\Span{x,y}^{\perp})) \cap \kappa(H) = 0$.
\end{compactenum}
\end{lemma}
\begin{proof}
The proofs of the two conclusions are similar; we will give the details for (a)
first and leave (b) to the reader.  Choose a symplectic basis
$\{a_1,b_1,\ldots,a_g,b_g\}$ for $H$ with $a_1 = x$.  We have $\omega = a_1 \wedge b_1 + \cdots + a_g \wedge b_g$.
Since $H = a_1^{\perp} \oplus \Span{b_1}$, we have
\begin{equation}
\label{eqn:linearalgebrasum}
\wedge^3 H = (\wedge^2 (a_1^{\perp})) \oplus (\wedge^3 (a_1^{\perp})),
\end{equation}
where the inclusion $\wedge^2 (a_1^{\perp}) \rightarrow \wedge^3 H$ of the first factor
takes $\theta \in \wedge^2 (a_1^{\perp})$ to $\theta \wedge b_1$.

Consider $h \in H$.  Write $h = c a_1 + d b_1 + h'$ with $c,d \in \Z$ and $h' \in \Span{a_2,b_2,\ldots,a_g,b_g}$.
Letting $\overline{\omega} = a_2 \wedge b_2 + \cdots + a_g \wedge b_g$, an easy calculation shows that
the expression for $\kappa(h)$ in terms of \eqref{eqn:linearalgebrasum} is
$(d \overline{\omega} + h' \wedge a_1, (c a_1 + h') \wedge \overline{\omega})$.
The first coordinate of this is $0$ if and only if $d = 0$ and $h' = 0$, and the lemma
follows.
\end{proof}

\begin{corollary}
\label{corollary:reducetoclosed}
Let $\Sigma_h^1 \hookrightarrow \Sigma_g$ be a clean embedding.  Then
the natural map $\wedge^3 \HH_1(\Sigma_h^1;\Z) \rightarrow (\wedge^3 \HH_1(\Sigma_g;\Z))/\HH_1(\Sigma_g;\Z)$
is injective.
\end{corollary}
\begin{proof}
Since the embedding is clean, we have $g>h$.  The corollary now follows from the second
conclusion of Lemma \ref{lemma:linearalgebra}.
\end{proof}

If $\Sigma_h^1 \hookrightarrow \Sigma_g^n$ is a clean embedding with $n \leq 1$,
then Lemma \ref{lemma:degeneratetorelli} implies that
$\TorelliRel{\Sigma_h^1}{\Sigma_g^n} = \Torelli{\Sigma_h^1}$.  We therefore have
two different Johnson homomorphisms on $\TorelliRel{\Sigma_h^1}{\Sigma_g^n}$, one whose codomain
is $\wedge^3 \HH_1(\Sigma_h^1;\Z)$ and one whose codomain is either $\wedge^3 \HH_1(\Sigma_g^n;\Z)$
if $n = 1$ or $(\wedge^3 \HH_1(\Sigma_g^n;\Z))/\HH_1(\Sigma_g^n;\Z)$ if $n=0$.  Thankfully
both of these Johnson homomorphisms have the same kernel.

\begin{lemma}
\label{lemma:noambiguity}
Let $\Sigma_h^1 \hookrightarrow \Sigma_g^n$ be a clean embedding with $n \leq 1$.  Set
$H = \HH_1(\Sigma_g^n;\Z)$.
\begin{compactitem}
\item If $n=1$, then $\tau_{\Sigma_h^1,\Sigma_g^n}:\TorelliRel{\Sigma_h^1}{\Sigma_g^n} \rightarrow \wedge^3 H$ factors as
\begin{equation}
\label{eqn:factor1}
\TorelliRel{\Sigma_h^1}{\Sigma_g^n} = \Torelli{\Sigma_h^1} \stackrel{\tau_{\Sigma_h^1}}{\longrightarrow} \wedge^3 \HH_1(\Sigma_h^1;\Z) \longrightarrow \wedge^3 H.
\end{equation}
\item If $n=0$, then $\tau_{\Sigma_h^1,\Sigma_g^n}:\TorelliRel{\Sigma_h^1}{\Sigma_g^n} \rightarrow (\wedge^3 H)/H$ factors as
\[\TorelliRel{\Sigma_h^1}{\Sigma_g^n} = \Torelli{\Sigma_h^1} \stackrel{\tau_{\Sigma_h^1}}{\longrightarrow} \wedge^3 \HH_1(\Sigma_h^1;\Z) \longrightarrow \wedge^3 H \longrightarrow (\wedge^3 H)/H.\]
\end{compactitem}
Also, in both cases we have 
$\JKer{\Sigma_h^1} = \JKerRel{\Sigma_h^1}{\Sigma_g^n}$.
\end{lemma}
\begin{proof}
The function with the factorization \eqref{eqn:factor1} has the same effect on separating
twists and bounding pair maps as $\tau_{\Sigma_h^1,\Sigma_g^n}$.  These generate $\TorelliRel{\Sigma_h^1}{\Sigma_g^n}$, 
so it must equal $\tau_{\Sigma_h^1,\Sigma_g^n}$.  A similar proof works for the case $n=0$.
As for the final claim, for $n=1$ it follows from the fact that the map
$\wedge^3 \HH_1(\Sigma_h^1;\Z) \rightarrow \wedge^3 \HH_1(\Sigma_g^n;\Z)$ is
injective, and for $n=0$ it follows from Corollary \ref{corollary:reducetoclosed}.
\end{proof}

%%%%%%%%%%%%%%%%%%%%%%%%%%%%%%%%%%%%%%%%%%%%%%%%%%%%%%%%%%%%%%%%%%%%%%%%%%%%%%%%%%%%%%%%%%%%
%%%%%%%%%%%%%%%%%%%%%%%%%%%%%%%%%%%%%%%%%%%%%%%%%%%%%%%%%%%%%%%%%%%%%%%%%%%%%%%%%%%%%%%%%%%%
%%%%%%%%%%%%%%%%%%%%%%%%%%%%%%%%%%%%%%%%%%%%%%%%%%%%%%%%%%%%%%%%%%%%%%%%%%%%%%%%%%%%%%%%%%%%
%%%%%%%%%%%%%%%%%%%%%%%%%%%%%%%%%%%%%%%%%%%%%%%%%%%%%%%%%%%%%%%%%%%%%%%%%%%%%%%%%%%%%%%%%%%%
%%%%%%%%%%%%%%%%%%%%%%%%%%%%%%%%%%%%%%%%%%%%%%%%%%%%%%%%%%%%%%%%%%%%%%%%%%%%%%%%%%%%%%%%%%%%

\section{Birman exact sequence for the Johnson kernel}
\label{section:birmanjohnson}

Let $S$ be a surface, let $S \hookrightarrow \Sigma_g^n$ be a clean embedding with $n \leq 1$,
and let $\beta$ be a boundary component of $S$.
Define $\PPK{S}{\Sigma_g^n}{\beta} = \PPM{S}{\beta} \cap \JKerRel{S}{\Sigma_g}$.
In this section, we construct an analogue of the Birman exact sequence
for the Johnson kernel; the ``differences'' between the Johnson kernels
on surfaces with different numbers of boundary components will be encoded
by the disc-pushing subgroups $\PPK{S}{\Sigma_g^n}{\beta}$.
Just like for the Torelli group, there are three cases (corresponding to
Theorems \ref{theorem:birmantorellione}, \ref{theorem:birmantorellisep}, 
and \ref{theorem:birmantorellinosep}).

%%%%%%%%%%%%%%%%%%%%%%%%%%%%%%%%%%%%%%%%%%%%%%%%%%%%%%%%%%%%%%%%%%%%%%%%%%%%%%%%%%%%%%%%%%%%
%%%%%%%%%%%%%%%%%%%%%%%%%%%%%%%%%%%%%%%%%%%%%%%%%%%%%%%%%%%%%%%%%%%%%%%%%%%%%%%%%%%%%%%%%%%%
%%%%%%%%%%%%%%%%%%%%%%%%%%%%%%%%%%%%%%%%%%%%%%%%%%%%%%%%%%%%%%%%%%%%%%%%%%%%%%%%%%%%%%%%%%%%
%%%%%%%%%%%%%%%%%%%%%%%%%%%%%%%%%%%%%%%%%%%%%%%%%%%%%%%%%%%%%%%%%%%%%%%%%%%%%%%%%%%%%%%%%%%%
%%%%%%%%%%%%%%%%%%%%%%%%%%%%%%%%%%%%%%%%%%%%%%%%%%%%%%%%%%%%%%%%%%%%%%%%%%%%%%%%%%%%%%%%%%%%

\subsection{One boundary}
\label{section:birmanjohnsonone}

The following theorem is the analogue for the Johnson kernel of Theorem
\ref{theorem:birmantorellione}.  It is very close to results of
Johnson from \cite{JohnsonKg}; see in particular \cite[\S 4]{JohnsonKg}.

\begin{theorem}
\label{theorem:birmanjohnsonone}
Let $S \hookrightarrow \Sigma_g^n$ be a clean embedding with $n \leq 1$ and let $\beta$ be a boundary
component of $S$.  Assume that the genus of $S$ is at least $2$ and that $\beta$ is the only
boundary component of $S$. Let $\hat{S}$ be the result of gluing a disc to $\beta$.  Then
there is an exact sequence
\[1 \longrightarrow \PPK{S}{\Sigma_g^n}{\beta} \longrightarrow \JKerRel{S}{\Sigma_g^n} \longrightarrow \JKer{\hat{S}} \longrightarrow 1.\]
Also, $\PPK{S}{\Sigma_g^n}{\beta} \subset \PPM{S}{\beta}$ is the pullback of $[\pi_1(\hat{S}),\pi_1(\hat{S})]$
under the map
\[\PPM{S}{\beta} \cong \pi_1(U\hat{S}) \longrightarrow \pi_1(\hat{S}).\]
\end{theorem}
\begin{proof}
By Lemma \ref{lemma:noambiguity}, we have $\JKerRel{S}{\Sigma_g^n} = \JKer{S}$.  The Johnson homomorphism
on a closed surface was defined precisely so that the surjection $\Torelli{S} \rightarrow \Torelli{\hat{S}}$
from Theorem \ref{theorem:birmantorellione} restricts to a surjection $\JKer{S} \rightarrow \JKer{\hat{S}}$ (see
the commutative diagram \eqref{eqn:johnsonhomodiagram}).  The indicated exact sequence follows.  As for the identification
of $\PPK{S}{\Sigma_g^n}{\beta}$, Theorem \ref{theorem:birmantorellione} says that
$\PPT{S}{\Sigma_g^n}{\beta} = \PPM{S}{\beta}$.  It then follows from the aforementioned identity
$\JKerRel{S}{\Sigma_g^n} = \JKer{S}$ that $\PPK{S}{\Sigma_g^n}{\beta} \subset \PPM{S}{\beta}$ is the kernel of 
the restriction of $\tau_S$ to $\PPM{S}{\beta}$.  Letting $H = \HH_1(\hat{S};\Z)$ and letting $\omega \in \wedge^2 H$
be the characteristic element, Lemma \ref{lemma:ptpushjohnsonclassical} says that this
restriction equals the composition
\[\PPM{S}{\beta} \cong \pi_1(U\hat{S}) \longrightarrow \pi_1(\hat{S}) \longrightarrow H \stackrel{\kappa}{\longrightarrow} \wedge^3 H,\]
where $\kappa: H \rightarrow \wedge^3 H$ is the map $\kappa(h) = h \wedge \omega$.  The map $\kappa$ is injective, so we conclude
that $\PPK{S}{\Sigma_g^n}{\beta} \subset \PPM{S}{\beta}$ is the kernel of the composition
\[\PPM{S}{\beta} \cong \pi_1(U\hat{S}) \longrightarrow \pi_1(\hat{S}) \longrightarrow H,\]
as desired.
\end{proof}

%%%%%%%%%%%%%%%%%%%%%%%%%%%%%%%%%%%%%%%%%%%%%%%%%%%%%%%%%%%%%%%%%%%%%%%%%%%%%%%%%%%%%%%%%%%%
%%%%%%%%%%%%%%%%%%%%%%%%%%%%%%%%%%%%%%%%%%%%%%%%%%%%%%%%%%%%%%%%%%%%%%%%%%%%%%%%%%%%%%%%%%%%
%%%%%%%%%%%%%%%%%%%%%%%%%%%%%%%%%%%%%%%%%%%%%%%%%%%%%%%%%%%%%%%%%%%%%%%%%%%%%%%%%%%%%%%%%%%%
%%%%%%%%%%%%%%%%%%%%%%%%%%%%%%%%%%%%%%%%%%%%%%%%%%%%%%%%%%%%%%%%%%%%%%%%%%%%%%%%%%%%%%%%%%%%
%%%%%%%%%%%%%%%%%%%%%%%%%%%%%%%%%%%%%%%%%%%%%%%%%%%%%%%%%%%%%%%%%%%%%%%%%%%%%%%%%%%%%%%%%%%%

\subsection{Separating boundary}
\label{section:birmanjohnsonsep}

The following theorem is the analogue
for the Johnson kernel of Theorem \ref{theorem:birmantorellisep}.

\begin{theorem}
\label{theorem:birmanjohnsonsep}
Let $S \hookrightarrow \Sigma_g^n$ be a clean embedding with $n \leq 1$ and let $\beta$ be a boundary
component of $S$.  Assume that the genus of $S$ is at least $2$, that $S$ has more than
one boundary component, and that $\beta$ separates $\Sigma_g^n$.
Let $S' \hookrightarrow S$ be a splitting surface for $\beta$ and let $\hat{S}$ be
the result of gluing a disc to $\beta$.  Then if all the boundary components of $S$ separate
$\Sigma_g^n$, we have
\[\JKerRel{S}{\Sigma_g^n} = \PPK{S}{\Sigma_g^n}{\beta} \rtimes \JKerRel{S'}{\Sigma_g^n}.\]
Also, even if the boundary components of $S$ other than $\beta$ do not separate $\Sigma_g^n$,
the group $\PPK{S}{\Sigma_g^n}{\beta}$ is the pullback
of $[\pi_1(\hat{S}),\pi_1(\hat{S})]_{\Sigma_g^n}$ under the map
\[\PPT{S}{\Sigma_g^n}{\beta} \cong \pi_1(U\hat{S}) \rightarrow \pi_1(\hat{S}).\]
\end{theorem}

\begin{remark}
The condition that all the boundary components of $S$ separate $\Sigma_g^n$ might seem
quite restrictive.  However, when we use the Birman exact sequence for the Johnson kernel,
the first step will be to cap off nonseparating boundary components
one at a time (using Theorem \ref{theorem:birmanjohnsonnosep} below) until only separating boundary
components remain, so Theorem \ref{theorem:birmanjohnsonsep} suffices for our purposes.
There are subtleties involved in formulating the correct statement without this condition.
For example, if $S$ and $S'$ and $x$ and $y$ and
$z$ are as in Figure \ref{figure:johnsonbirmansep}.a, then $T_x T_y^{-1} \in \PPT{S}{\Sigma_g}{\beta} \setminus \JKerRel{S}{\Sigma_g}$
and $T_y T_z^{-1} \in \TorelliRel{S'}{\Sigma_g} \setminus \JKerRel{S'}{\Sigma_g}$, 
but $(T_x T_y^{-1})(T_y T_z^{-1}) = T_x T_z^{-1} \in \JKerRel{S}{\Sigma_g}$.
\end{remark}

\Figure{figure:johnsonbirmansep}{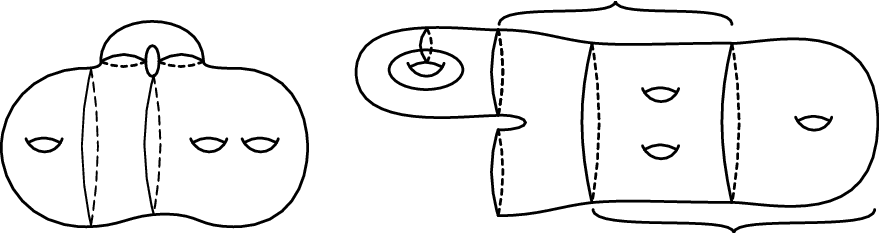}
{a. $S$ is the subsurface to the right of $\beta$ and below $x$ and $z$.  Considered as curves in $S$, the curves
$x$ and $y$ become isotopic when a disc is glued to $\beta$, so the $\Sigma_4$-bounding pair map $T_x T_y^{-1}$ lies
in $\PPT{S}{\Sigma_4}{\beta}$.  
\CaptionSpace
b. The subsurfaces from the last part of the proof of Theorem \ref{theorem:birmanjohnsonsep}.}

The proof of Theorem \ref{theorem:birmanjohnsonsep} requires the following lemma.

\begin{lemma}
\label{lemma:semidirect}
Let $f:G_1 \rightarrow G_2$ be a group homomorphism.  Assume that
$G_1 = K \rtimes Q$.  Then $\ker(f) = \ker(f|_K) \rtimes \ker(f|_Q)$ if and only if $f(K) \cap f(Q) = 1$.
\end{lemma}
\begin{proof}
Trivial.
\end{proof}

\begin{proof}[{Proof of Theorem \ref{theorem:birmanjohnsonsep}}]
If $n=1$, then by 
Lemmas \ref{lemma:degeneratetorelli}(b) and \ref{lemma:relativecommutatorunchanged},
we can replace $\Sigma_g^n$ with $\Sigma_{g+1}$ without changing the truth
of the theorem, so without loss of generality we can assume that $n=0$.  Next,
another application of
Lemmas \ref{lemma:degeneratetorelli}(b) and \ref{lemma:relativecommutatorunchanged}
shows that we can replace $\Sigma_g$ with the component of $\Sigma_g$ cut along
$\beta$ that contains $S$ without changing the truth of the theorem, so
without loss of generality we can assume that $n=1$ and that $\beta = \partial \Sigma_g^n$.

Set $H = \HH_1(\Sigma_g^1;\Z)$ and let $\omega \in \wedge^2 H$ be the
characteristic element.  Define $\kappa:H \rightarrow \wedge^3 H$ via the formula
$\kappa(h) = h \wedge \omega$.  Theorem \ref{theorem:birmantorellisep} says that
$\PPT{S}{\Sigma_g^1}{\beta} = \PPM{S}{\beta} \cong \pi_1(U\hat{S})$.
By Lemma \ref{lemma:ptpushjohnsonclassical}, the restriction
of $\tau_{S,\Sigma_g^1}$ to $\PPM{S}{\beta}$ equals the composition
\[\PPM{S}{\beta} \hookrightarrow \PPM{\Sigma_g^1}{\beta} \cong \pi_1(U\Sigma_g) \rightarrow \pi_1(\Sigma_g) \rightarrow H \stackrel{\kappa}{\longrightarrow} \wedge^3 H,\]
where we are identifying $H$ with $\HH_1(\Sigma_g;\Z)$ in the obvious way.  This
composition equals
\begin{equation}
\label{eqn:factorjohnsonsep}
\PPM{S}{\beta} \cong \pi_1(U\hat{S}) \rightarrow \pi_1(\hat{S}) \rightarrow \HH_1(\hat{S};\Z) \stackrel{\sigma_{S,\Sigma_g^1,\beta}}{\longrightarrow} H \stackrel{\kappa}{\longrightarrow} \wedge^3 H.
\end{equation}
Since $\kappa$ is injective, we see that the kernel of the restriction of
$\tau_{S,\Sigma_g^1}$ to $\PPM{S}{\beta}$ equals the kernel of the composition
\[\PPM{S}{\beta} \cong \pi_1(U\hat{S}) \rightarrow \pi_1(\hat{S}) \rightarrow \HH_1(\hat{S};\Z) \stackrel{\sigma_{S,\Sigma_g^1,\beta}}{\longrightarrow} H.\]
By definition, this is the pullback to $\PPM{S}{\beta}$ of 
$[\pi_1(\hat{S}),\pi_1(\hat{S})]_{\Sigma_g^1}$, as claimed.

It remains to prove that if all the boundary components of $S$ separate $\Sigma_g^1$, then
\[\JKerRel{S}{\Sigma_g^1} = \PPK{S}{\Sigma_g^1}{\beta} \rtimes \JKerRel{S'}{\Sigma_g^1}.\]
By Theorem \ref{theorem:birmantorellisep}, we have
\[\TorelliRel{S}{\Sigma_g^1} = \PPT{S}{\Sigma_g^1}{\beta} \rtimes \TorelliRel{S'}{\Sigma_g^1}.\]
Lemma \ref{lemma:semidirect} thus implies that it is enough to prove that
\begin{equation}
\label{eqn:birmantorelliseptoprove}
\tau_{S,\Sigma_g^1}(\PPT{S}{\Sigma_g^1}{\beta}) \cap \tau_{S,\Sigma_g^1}(\TorelliRel{S'}{\Sigma_g^1}) = 0.
\end{equation}
Let $\{\beta,\beta',\beta''\}$ be the boundary components of $S \setminus \Interior(S')$,
ordered such that $\beta'' = S \cap S'$.  Let $M$ be the component
of $\Sigma_g^1$ cut along $\beta'$ that does not contain $S$.  Also,
let $S''$ be the component of $\Sigma_g^1$ cut along $\beta''$ that contains
$S'$ (see Figure \ref{figure:johnsonbirmansep}(b)).  Using our assumption
that all the boundary components of $S$ separate $\Sigma_g^1$, we see that
$M$ and $S''$ are disjoint surfaces with one boundary component whose
genera are positive.  We can thus find oriented simple closed curves
$x$ and $y$ in $M$ that intersect once, and
$\HH_1(S'';\Z) \subset \langle [x], [y] \rangle^{\perp}$.
Lemma \ref{lemma:linearalgebra}(b) therefore implies that
\[\left(H \wedge \omega\right) \cap \left(\wedge^3 \HH_1(S'';\Z)\right) = 0.\]
Lemma \ref{lemma:noambiguity} implies that
$\tau_{S,\Sigma_g^1}(\TorelliRel{S'}{\Sigma_g^1}) \subset \wedge^3 \HH_1(S'';\Z)$,
and it is immediate from \eqref{eqn:factorjohnsonsep} that
$\tau_{S,\Sigma_g^1}(\PPT{S}{\Sigma_g^1}{\beta}) \subset H \wedge \omega$.
The desired equation \eqref{eqn:birmantorelliseptoprove} follows.
\end{proof}

%%%%%%%%%%%%%%%%%%%%%%%%%%%%%%%%%%%%%%%%%%%%%%%%%%%%%%%%%%%%%%%%%%%%%%%%%%%%%%%%%%%%%%%%%%%%
%%%%%%%%%%%%%%%%%%%%%%%%%%%%%%%%%%%%%%%%%%%%%%%%%%%%%%%%%%%%%%%%%%%%%%%%%%%%%%%%%%%%%%%%%%%%
%%%%%%%%%%%%%%%%%%%%%%%%%%%%%%%%%%%%%%%%%%%%%%%%%%%%%%%%%%%%%%%%%%%%%%%%%%%%%%%%%%%%%%%%%%%%
%%%%%%%%%%%%%%%%%%%%%%%%%%%%%%%%%%%%%%%%%%%%%%%%%%%%%%%%%%%%%%%%%%%%%%%%%%%%%%%%%%%%%%%%%%%%
%%%%%%%%%%%%%%%%%%%%%%%%%%%%%%%%%%%%%%%%%%%%%%%%%%%%%%%%%%%%%%%%%%%%%%%%%%%%%%%%%%%%%%%%%%%%

\subsection{Nonseparating boundary}
\label{section:birmanjohnsonnosep}

The following theorem is the analogue
for the Johnson kernel of Theorem \ref{theorem:birmantorellinosep}.

\begin{theorem}
\label{theorem:birmanjohnsonnosep}
Let $S \hookrightarrow \Sigma_g^n$ be a clean embedding with $n \leq 1$ and let $\beta$ be a boundary
component of $S$.  Assume that the genus of $S$ is at least $2$ and
that $\beta$ is a nonseparating curve on $\Sigma_g^n$.
Let $S' \hookrightarrow S$ be a $\Sigma_g^n$-splitting surface for $\beta$ and let $\hat{S}$ be
the result of gluing a disc to $\beta$.  Then
\[\JKerRel{S}{\Sigma_g^n} = \PPK{S}{\Sigma_g^n}{\beta} \rtimes \JKerRel{S'}{\Sigma_g^n}.\]
Also, the following hold.  Let $C$ be the component of $\Sigma_g^n \setminus \Interior(S)$ adjacent to $\beta$.
\begin{compactitem}
\item If $C$ is not an annulus, then
$\PPK{S}{\Sigma_g^n}{\beta} = [\pi,[\pi,\pi]_{\Sigma_g^n}] \subset [\pi,\pi]_{\Sigma_g^n}$.
\item If $C$ is an annulus and $\beta' \subset S$ is its other boundary component,
then $\PPK{S}{\Sigma_g^n}{\beta}$ is generated by $[\pi,[\pi,\pi]_{\Sigma_g^n}]$ together with
a simple closed curve $\delta \in [\pi,\pi]_{\Sigma_g^n}$ which is freely homotopic to the component of $\partial \hat{S}$ corresponding
to $\beta'$ (see Figures \ref{figure:johnsondiscnosep1}.a--b).
\end{compactitem}
\end{theorem}

\Figure{figure:johnsondiscnosep1}{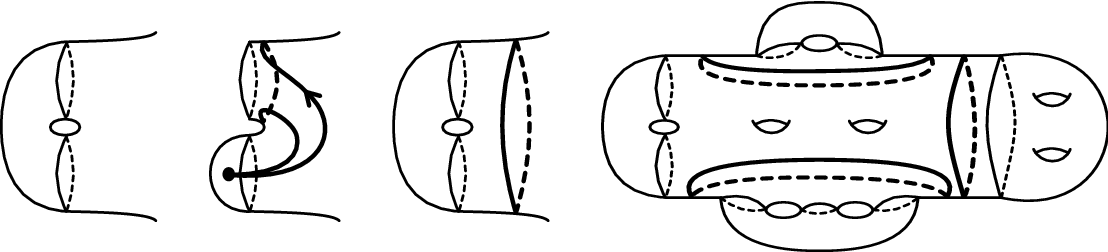}{
a--c. Why an extra generator is needed in Theorem \ref{theorem:birmanjohnsonnosep} if $C$ is an annulus \CaptionSpace
d. $S \cong \Sigma_2^{8}$ is the subsurface in the ``middle'' of $\Sigma_8$.}

\begin{remark}
The reason for the additional generator of $\PPK{S}{\Sigma_g^n}{\beta}$ when $C$ is an annulus with boundary
components $\beta$ and $\beta'$ is as follows.
As shown in Figure \ref{figure:johnsondiscnosep1}.a--c, if $\delta \in [\pi,\pi]_{\Sigma_g^n}$
is a simple closed curve freely homotopic to $\beta'$, then using the explicit formula from \S \ref{section:torellibasic}
we have (up to signs) $\PtPshHat{\delta} = T_{\beta} T_{\beta'}^{-1} T_{\epsilon}$, where $T_{\epsilon}$ is a $\Sigma_g^n$-separating
twist.  Thus
\[\tau_{S,\Sigma_g^n}(\PtPshHat{\delta}) = \tau_{S,\Sigma_g^n}(T_{\beta} T_{\beta'}^{-1}) + \tau_{S,\Sigma_g^n}(T_{\epsilon}) = \tau_{S,\Sigma_g^n}(T_{\beta} T_{\beta'}^{-1}).\]
However, $\beta$ and $\beta'$ are isotopic in $\Sigma_g^n$, so $T_{\beta} T_{\beta'}^{-1}$ is in the kernel of the map
$\TorelliRel{S}{\Sigma_g^n} \rightarrow \Torelli{\Sigma_g^n}$ and thus certainly in the kernel of $\tau_{S,\Sigma_g^n}$.
\end{remark}

\begin{proof}[{Proof of Theorem \ref{theorem:birmanjohnsonnosep}}]
If $n=1$, then by
Lemmas \ref{lemma:degeneratetorelli}(b) and \ref{lemma:relativecommutatorunchanged},
we can replace $\Sigma_g^n$ with $\Sigma_{g+1}$ without changing the truth
of the theorem, so without loss of generality we can assume that $n=0$.

The first thing we must do is to construct an analogue for our situation of the factorization in
Lemma \ref{lemma:ptpushjohnsonclassical}.  To simplify our notation, we make the following definitions.
\begin{compactitem}
\item $H = \HH_1(\Sigma_g;\Z)$.
\item $\pi = \pi_1(\hat{S})$.
\item $\sigma = \sigma_{S,\Sigma_g,\beta}$ (this is the subsurface homology map defined in
\S \ref{section:birmantorelli}).
\end{compactitem}
Let $V \subset H/\Span{[\beta]}$ be the image of 
$\sigma: \HH_1(\hat{S};\Z) \rightarrow H / \Span{[\beta]}$ and let
$K = \ker(\sigma)$, so we have a short exact sequence
\[0 \longrightarrow K \longrightarrow \HH_1(\hat{S};\Z) \stackrel{\sigma}{\longrightarrow} V \longrightarrow 0.\]
Let the connected components of $\Sigma_g \setminus \Interior(S)$ be $X_0,\ldots,X_r$, where
$X_0$ is the component containing $\beta$.  Orienting the connected components of $\partial X_i \subset S$
in the usual way, the $\Sigma_g$-homology class of $\partial X_i$ is $0$, so we obtain an element
$[\partial X_i]' \in \ker(\HH_1(S;\Z) \rightarrow \HH_1(\Sigma_g;\Z))$.  Let $[\partial X_i] \in K$
be the image of $[\partial X_i]' \in \HH_1(S;\Z)$.  Then $K$ is generated by $[\partial X_0],\ldots,[\partial X_r]$,
and the only relation between the $[\partial X_i]$ is that $[\partial X_0] + \cdots + [\partial X_r] = 0$ (see
Figure \ref{figure:johnsondiscnosep1}.d).
The upshot is that $K \cong \Z^r$ with basis $\{[\partial X_1],\ldots,[\partial X_r]\}$.  We
now choose $\gamma_1,\ldots,\gamma_r \in \pi$ with the following properties
(see Figure \ref{figure:johnsondiscnosep1}.d).  
\begin{compactitem}
\item For $1 \leq i \leq r$, the curve $\gamma_i$ is a $\Sigma_g$-bounding curve surrounding $X_i$ (we recall that
$\Sigma_g$-bounding curves and torus-bounding curves were defined towards the end of \S \ref{section:torellibasic}),
oriented such that $[\gamma_i] = [\partial X_i]$.
\item For $1 \leq i \leq r$, the curve $\gamma_i$ is freely homotopic to an unbased simple closed
curve $\eta_i$ such that the curves $\{\eta_1,\ldots,\eta_r\}$ are pairwise disjoint.
\end{compactitem}
Observe that $\gamma_i \in [\pi,\pi]_{\Sigma_g}$.  We then have the following.

\BeginClaims
\begin{claims}
\label{claim:buildphi}
There is a surjective homomorphism $\phi: [\pi,\pi]_{\Sigma_g} \rightarrow K \oplus \wedge^2 V$ with the following
properties.
\begin{compactitem}
\item For $1 \leq i \leq r$, we have $\phi(\gamma_i) = [\partial X_i] \in K$.
\item For $x,y \in \pi$, we have $\phi([x,y]) = \sigma([x]) \wedge \sigma([y])$, where $[x],[y] \in \HH_1(\hat{S};\Z)$
are the associated homology classes.
\item $\ker(\phi) = [\pi,[\pi,\pi]_{\Sigma_g}]$.
\end{compactitem}
\end{claims}
\begin{proof}[Proof of claim]
We have a short exact sequence
\[1 \longrightarrow [\pi,\pi]_{\Sigma_g} \longrightarrow \pi \longrightarrow V \longrightarrow 1.\]
Associated to this is a $5$-term exact sequence in group homology (see \cite[Corollary VII.6.4]{BrownCohomology}).  
The group $\pi$ is free, so $\HH_2(\pi;\Z) = 0$.  Also, $\HH_2(V;\Z) \cong \wedge^2 V$.  
The $5$-term sequence is thus of the form
\[0 \longrightarrow \wedge^2 V \longrightarrow (\HH_1([\pi,\pi]_{\Sigma_g};\Z))_{\pi} \longrightarrow \HH_1(\pi;\Z) \longrightarrow V \longrightarrow 0.\]
Here $(\HH_1([\pi,\pi]_{\Sigma_g};\Z))_{\pi}$ are the coinvariants of the action of $\pi$ on $\HH_1([\pi,\pi]_{\Sigma_g};\Z)$,
i.e.\ the largest quotient of $\HH_1([\pi,\pi]_{\Sigma_g};\Z)$
on which the action of $\pi$ induced by the conjugation action of $\pi$ on $[\pi,\pi]_{\Sigma_g}$ is trivial.
Chasing through the definitions,
we have $(\HH_1([\pi,\pi]_{\Sigma_g};\Z))_{\pi} = [\pi,\pi]_{\Sigma_g}/[\pi,[\pi,\pi]_{\Sigma_g}]$.  Also, the kernel
of the map $\HH_1(\pi;\Z) \rightarrow V$ is $K$.  We therefore obtain a short exact sequence
\[0 \longrightarrow \wedge^2 V \longrightarrow [\pi,\pi]_{\Sigma_g}/[\pi,[\pi,\pi]_{\Sigma_g}] \longrightarrow K \longrightarrow 0.\]
Since $K$ is a free abelian group, this exact sequence splits; in fact, there is a splitting that takes $[\partial X_i] \in K$ 
to the image
of $\gamma_i \in [\pi,\pi]_{\Sigma_g}$ in $[\pi,\pi]_{\Sigma_g}/[\pi,[\pi,\pi]_{\Sigma_g}]$.  From this splitting, we get an isomorphism
\[\phi' : [\pi,\pi]_{\Sigma_g}/[\pi,[\pi,\pi]_{\Sigma_g}] \stackrel{\cong}{\longrightarrow} K \oplus \wedge^2 V.\]
The desired map $\phi : [\pi,\pi]_{\Sigma_g} \rightarrow K \oplus \wedge^2 V$ is then the composition of $\phi'$ with the projection
$[\pi,\pi]_{\Sigma_g} \rightarrow [\pi,\pi]_{\Sigma_g}/[\pi,[\pi,\pi]_{\Sigma_g}]$.
\end{proof}

For $1 \leq i \leq r$, let $Y_i$ be the component of $\Sigma_g$ cut along $\eta_i$ that contains $X_i$.  Thus
$Y_i$ is a surface with one boundary component and we have an injection $\HH_1(Y_i;\Z) \hookrightarrow \HH_1(\Sigma_g;\Z)$;
let $\omega_i \in \wedge^2 H$ be the image of the characteristic element of $\wedge^2 \HH_1(Y_i;\Z)$.  
Define a homomorphism $\psi_1:K \rightarrow \wedge^3 H$ via the formula $\psi_1([\partial X_i]) = [\beta] \wedge \omega_i$.
Also, define a homomorphism $\psi_2: \wedge^2 V \rightarrow \wedge^3 H$ via the formula
$\psi_2(v_1 \wedge v_2) = [\beta] \wedge \tilde{v}_1 \wedge \tilde{v}_2$, where $\tilde{v}_i \in H$ is any lift
of $v_i \in V \subset H/\Span{[\beta]}$ (observe that this is well-defined).  Finally, define
a homomorphism $\psi: K \oplus \wedge^2 V \rightarrow \wedge^3 H$ via the formula 
$\psi(k,v_1 \wedge v_2) = \psi_1(k) + \psi_2(v_1 \wedge v_2)$.  We 
then have the following, which is the analogue for our situation of Lemma \ref{lemma:ptpushjohnsonclassical}.
The first isomorphism in this claim follows from Theorem \ref{theorem:birmantorellinosep}.

\begin{claims}
\label{claim:factortau}
The restriction of $\tau_{S,\Sigma_g}:\TorelliRel{S}{\Sigma_g} \rightarrow (\wedge^3 H)/H$ to $\PPT{S}{\Sigma_g}{\beta}$
factors as
\[\PPT{S}{\Sigma_g}{\beta} \cong [\pi,\pi]_{\Sigma_g} \stackrel{\phi}{\longrightarrow} K \oplus \wedge^2 V \stackrel{\psi}{\longrightarrow} \wedge^3 H \longrightarrow (\wedge^3 H)/H.\]
\end{claims}
\begin{proof}[Proof of claim]
Let $\nu : [\pi,\pi]_{\Sigma_g} \rightarrow (\wedge^3 H)/H$ be the indicated composition.
The group $[\pi,\pi]_{\Sigma_g}$ is generated by $[\pi,\pi]$ together with $\gamma_1,\ldots,\gamma_{r}$, and
Lemma \ref{lemma:commgen} says that $[\pi,\pi]$ is generated by torus-bounding curves.
It is enough therefore to prove that $\nu(\lambda) = \tau_{S,\Sigma_g}(\PtPshHat{\lambda})$ for
$\lambda$ either one of the $\gamma_i$ or a torus-bounding curve.  We will give
the details for the latter case; the former can be handled similarly.

\Figure{figure:johnsondiscnosep2}{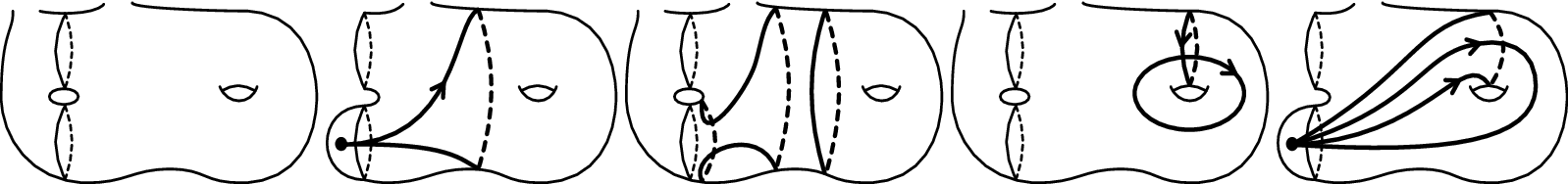}{
a. $S$ is a subsurface of $\Sigma_g$ and $\beta$ is a boundary component of $S$ that does not separate $\Sigma_g$ \CaptionSpace
b. $\lambda \in \pi$ is a torus-bounding curve \CaptionSpace
c. $\PtPshHat{\lambda} = T_{\tilde{\lambda}_1} (T_{\beta} T_{\tilde{\lambda}_2}^{-1})$ \CaptionSpace
d. $\tau_{S,\Sigma_g}(T_{\tilde{\lambda}_1} (T_{\beta} T_{\tilde{\lambda}_2}^{-1})) = \tau_{S,\Sigma_g}(T_{\tilde{\lambda}_1}) + \tau_{S,\Sigma_g}(T_{\beta} T_{\tilde{\lambda}_2}^{-1}) = 0 + [\beta] \wedge [A] \wedge [B] = [\beta] \wedge [A] \wedge [B]$ \CaptionSpace
e. $\lambda = [a,b]$}

Let $\lambda \in [\pi,\pi]$ be torus-bounding curve.  Replacing
$\lambda$ by its inverse if necessary, we can assume that the one-holed torus lies on its right; see Figures \ref{figure:johnsondiscnosep2}.a--b.
Letting $\tilde{\lambda}_1$ and $\tilde{\lambda}_2$ be the curves depicted in Figure \ref{figure:johnsondiscnosep2}.c, the explicit
formulas in \S \ref{section:torellibasic} show that $\PtPshHat{\lambda} = T_{\tilde{\lambda}_1} (T_{\beta} T_{\tilde{\lambda}_2}^{-1})$.  Letting
$A$ and $B$ be the curves shown in Figure \ref{figure:johnsondiscnosep2}.d, the formulas in \S \ref{section:johnsonhomomorphism}
show that $\tau_{S,\Sigma_g}(\PtPshHat{\lambda})$ equals
\[\tau_{S,\Sigma_g}(T_{\tilde{\lambda}_1}) + \tau_{S,\Sigma_g}(T_{\beta} T_{\tilde{\lambda}_2}^{-1}) = 0 + [\beta] \wedge [A] \wedge [B] = [\beta] \wedge [A] \wedge [B] \in (\wedge^3 H)/H\]
Letting $a,b \in \pi$ be the curves in Figure \ref{figure:johnsondiscnosep2}.e, we have $\lambda = [a,b]$.  Regarding $A$ and $B$
as curves in $\hat{S}$ via the inclusion $S \hookrightarrow \hat{S}$, the curve $a$ is freely homotopic to $A$ and the curve $b$
is freely homotopic to $B$.  Tracing through the definitions, we see that $\nu(\lambda) = [\beta] \wedge [A] \wedge [B] \in (\wedge^3 H)/H$, as desired.
\end{proof}

We are now in a position to prove the first assertion of the theorem, namely the following claim.

\Figure{figure:johnsonbirmannosep}{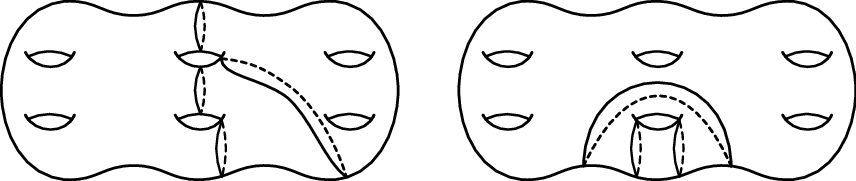}
{a. $S$ is the subsurface to the right of the three ``vertical'' curves.  $S'$ is a $\Sigma_6$-splitting surface for $S$. \CaptionSpace
b. $\mathcal{A} \subset \Sigma_6 \setminus \Interior(S)$ is an annulus one of whose boundary components is $\beta$.  We have
$M = \Sigma_6 \setminus \Interior(\mathcal{A})$, and $M' \subset M$ is a $\Sigma_6$-splitting surface for $\beta$ satisfying $S' \subset M'$.}

\begin{claims}
\label{claim:semidirect}
We have $\JKerRel{S}{\Sigma_g} = \PPK{S}{\Sigma_g}{\beta} \rtimes \JKerRel{S'}{\Sigma_g}$.
\end{claims}
\begin{proof}[Proof of claim]
Theorem \ref{theorem:birmantorellinosep} implies that
\[\TorelliRel{S}{\Sigma_g} = \PPT{S}{\Sigma_g}{\beta} \rtimes \TorelliRel{S'}{\Sigma_g}.\]
By Lemma \ref{lemma:semidirect}, the claim is equivalent to
\[\tau_{S,\Sigma_g}(\PPT{S}{\Sigma_g}{\beta}) \cap \tau_{S,\Sigma_g}(\TorelliRel{S'}{\Sigma_g}) = 0,\]
which we now prove.
Let $\mathcal{A}$ be an annulus in $\Sigma_g \setminus \Interior(S)$ one of whose boundary components is $\beta$ and let $M = \Sigma_g \setminus \Interior(\mathcal{A})$.  We
have $S \subset M$, and we can choose a $\Sigma_g$-splitting surface $M'$ for $\beta$ in $M$ such that $S' \subset M'$ (see Figure \ref{figure:johnsonbirmannosep}).  We have
$$\tau_{S,\Sigma_g}(\PPT{S}{\Sigma_g}{\beta}) \subset \tau_{M,\Sigma_g}(\PPT{M}{\Sigma_g}{\beta}) \quad \text{and} \quad \tau_{S,\Sigma_g}(\TorelliRel{S'}{\Sigma_g}) \subset \tau_{M,\Sigma_g}(\TorelliRel{M'}{\Sigma_g}),$$
so it suffices to prove that
\begin{equation}
\label{eqn:noseptoprove2}
\tau_{M,\Sigma_g}(\PPT{M}{\Sigma_g}{\beta}) \cap \tau_{M,\Sigma_g}(\TorelliRel{M'}{\Sigma_g}) = 0.
\end{equation}

Set $H_{M'} = \HH_1(M';\Z) \subset H$.  Let $\rho : \wedge^3 H \rightarrow (\wedge^3 H)/H$
be the projection.
Claim \ref{claim:factortau} implies that
\[\tau_{M,\Sigma_g}(\PPT{M}{\Sigma_g}{\beta}) = \rho([\beta] \wedge (\wedge^2 H_{M'})) \subset (\wedge^3 H)/H.\]
Also, Lemma \ref{lemma:noambiguity} implies that
\[\tau_{M,\Sigma_g}(\TorelliRel{M'}{\Sigma_g}) = \rho(\wedge^3 H_{M'}) \subset (\wedge^3 H)/H.\]
Equation \eqref{eqn:noseptoprove2} therefore asserts that
\begin{equation}
\label{eqn:noseptoprove3}
\rho([\beta] \wedge (\wedge^2 H_{M'})) \cap  \rho(\wedge^3 H_{M'}) = 0.
\end{equation}
We have that $([\beta] \wedge (\wedge^2 H_{M'})) \cap \wedge^3 H_{M'} = 0$ (as subsets of
$\wedge^3 H$).  In fact, we have
\begin{equation}
\label{eqn:directsum1}
\wedge^3 [\beta]^{\perp} = ([\beta] \wedge (\wedge^2 H_{M'})) \oplus \wedge^3 H_{M'};
\end{equation}
here the orthogonal complement is taken with respect to the intersection form.  The first
conclusion of Lemma \ref{lemma:linearalgebra} says that the kernel of $\rho|_{\wedge^3 [\beta]^{\perp}}$
lies entirely in $[\beta] \wedge (\wedge^2 H_{M'})$.  Therefore, the decomposition
\eqref{eqn:directsum1} projects to
\begin{equation}
\label{eqn:remember}
\rho(\wedge^3 [\beta]^{\perp}) = \rho([\beta] \wedge (\wedge^2 H_{M'})) \oplus \rho(\wedge^3 H_{M'}).
\end{equation}
In particular, \eqref{eqn:noseptoprove3} holds.
\end{proof}

For the second assertion of the theorem (concerning $\PPK{S}{\Sigma_g}{\beta}$), we need the following.

\begin{claims}
\label{claim:psiinjective}
The homomorphism $\psi$ is injective.
\end{claims}
\begin{proof}[Proof of claim]
The map $\psi|_{\wedge^2 V}$ is injective with image $[\beta] \wedge (\wedge^2 V)$.
To prove that $\psi$ itself is injective, it is enough to prove that
the elements $[\beta] \wedge \omega_1,\ldots,[\beta] \wedge \omega_{r} \in \wedge^3 H$ are linearly
independent modulo $[\beta] \wedge (\wedge^2 V)$.

Recall that $\omega_i$ is the characteristic element of $\wedge^2 \HH_1(Y_i;\Z)$, where $Y_i$ is the subsurface
of $\Sigma_g$ to the right of $\eta_i$.  Choose
a symplectic basis $B$ for $H$ with the following properties.
\begin{compactitem}
\item For $1 \leq i \leq r$, there is some subset of $B$ forming a symplectic basis for $\HH_1(Y_i;\Z)$.
\item $[\beta] \in B$.
\item There is a subset of $B$ that projects to a basis for $V \subset H / \Span{[\beta]}$.
\end{compactitem}
Choose a total ordering $\prec$ on $B$ such that $[\beta]$ is a minimal element.  The set
$B_{\wedge} = \{\text{$x \wedge y \wedge z$ $|$ $x,y \in B$, $x \prec y \prec z$}\}$ is a basis for $\wedge^3 H$.
Define
\[B_{\wedge}' = \{\text{$[\beta] \wedge y \wedge z$ $|$ $y,z \in B$, $[\beta] \prec y \prec z$}\} \subset B_{\wedge}.\]
By assumption, there is a subset $B_{\wedge}''$ of $B_{\wedge}'$ which forms a basis for $[\beta] \wedge (\wedge^2 V)$.
Also, for all $1 \leq i \leq r$, the element $[\beta] \wedge \omega_i$ is in the span of $B_{\wedge}'$.
Finally, for all $1 \leq i \leq r$, there exists some $v_i \in B_{\wedge}'$ with the following properties.
\begin{compactitem}
\item In the expansion of $[\beta] \wedge \omega_i$ in the basis $B_{\wedge}'$, the
coordinate of $v_i$ is nonzero.
\item $v_i \notin B_{\wedge}''$.
\item $v_i \neq v_j$ for $i \neq j$.
\end{compactitem}
It follows that the elements $[\beta] \wedge \omega_i$ are linearly independent modulo $[\beta] \wedge (\wedge^2 V)$, 
and we are done.
\end{proof}

The following claim is the remaining part of the theorem

\begin{claims}
\label{claim:generatepush}
Let $C$ be the component of $\Sigma_g^n \setminus \Interior(S)$ adjacent to $\beta$.
\begin{compactitem}
\item If $C$ is not an annulus, then
$\PPK{S}{\Sigma_g^n}{\beta} = [\pi,[\pi,\pi]_{\Sigma_g^n}] \subset [\pi,\pi]_{\Sigma_g^n}$.
\item If $C$ is an annulus and $\beta' \subset S$ is its other boundary component,
then $\PPK{S}{\Sigma_g^n}{\beta}$ is generated by $[\pi,[\pi,\pi]_{\Sigma_g^n}]$ together with
a simple closed curve $\delta \in [\pi,\pi]_{\Sigma_g^n}$ which is freely homotopic to the component of $\partial \hat{S}$ corresponding
to $\beta'$ (see Figures \ref{figure:johnsondiscnosep1}.a--b).
\end{compactitem}
\end{claims}
\begin{proof}[Proof of claim]
Let $\tau_1 : \PPT{S}{\Sigma_g}{\beta} \rightarrow \wedge^3 H$ be the composition
\[\PPT{S}{\Sigma_g}{\beta} \cong [\pi,\pi]_{\Sigma_g} \stackrel{\phi}{\longrightarrow} K \oplus \wedge^2 V \stackrel{\psi}{\longrightarrow} \wedge^3 H\]
and let $\tau_2 : \wedge^3 H \rightarrow (\wedge^3 H)/H$ be the projection.  Claim \ref{claim:factortau} 
says that $\tau_{S,\Sigma_g}|_{\PPT{S}{\Sigma_g}{\beta}} = \tau_2 \circ \tau_1$.
Claims \ref{claim:buildphi} and \ref{claim:psiinjective} show that $\Ker(\tau_1) = [\pi,[\pi,\pi]_{\Sigma_g}]$. 

If the component $C$ of $\Sigma_g \setminus \Interior(S)$
adjacent to $\beta$ is not an annulus, then we can find a simple closed curve $\zeta$ in $C$ such that
$\zeta$ is not homologous to $\beta$.  The image of $\tau_1$ lies in $\wedge^3 \Span{[\beta],[\zeta]}^{\bot}$, and the second
conclusion of Lemma \ref{lemma:linearalgebra} implies that $\tau_2|_{\wedge^3 \Span{[\beta],[\zeta]}^{\bot}}$
is injective.  We conclude that the kernel of $\tau_2 \circ \tau_1$ is
$[\pi,[\pi,\pi]_{\Sigma_g}]$, as desired.

If $C$ is an annulus, then we can cannot find such a $\zeta$.  In this case,
the image of $\tau_1$ lies in $\wedge^3 \Span{[\beta]}^{\bot}$,
and the first conclusion of Lemma \ref{lemma:linearalgebra} says that
$\Ker(\tau_2|_{\wedge^3 \Span{[\beta]}^{\bot}})$ is spanned by $[\beta] \wedge \omega$, where $\omega \in \wedge^2 H$
is the characteristic element.  Letting $\delta \in [\pi,\pi]_{\Sigma_g}$ be the element described in
the theorem, an easy calculation shows that $\tau_1(\delta) = \pm [\beta] \wedge \omega$.
We conclude that the kernel of $\tau_2 \circ \tau_1$ is generated by $[\pi,[\pi,\pi]_{\Sigma_g}]$
and $\delta$, as desired.
\end{proof}

This concludes the proof of Theorem \ref{theorem:birmanjohnsonnosep}.
\end{proof}

In the course of the above proof, 
we proved the following lemma (see in particular equation \eqref{eqn:remember}), which we will need later.

\begin{lemma}
\label{lemma:nosepdiscpushcalc}
For some $g \geq 3$, let $\beta$ be a nonseparating simple closed curve on $\Sigma_g$.  Let $\mathcal{A} \subset \Sigma_g$ be an annulus
one of whose boundary components is $\beta$, let $M = \Sigma_g \setminus \Interior(A)$, and let $M' \subset M$ be a $\Sigma_g$-splitting
surface for $\beta$.  Define $H = \HH_1(\Sigma_g;\Z)$ and $H_{M'} = \HH_1(M';\Z) \subset H$.  Then
\[\tau_{M,\Sigma_g}(\TorelliRel{M}{\Sigma_g}) \cong ((\wedge^2 H_{M'})/\mathbb{Z}) \oplus (\wedge^3 H_{M'}) \subset (\wedge^3 H)/H,\]
where $((\wedge^2 H_{M'})/\mathbb{Z})$ is the image of $\PPT{M}{\Sigma_g}{\beta} \subset \TorelliRel{M}{\Sigma_g}$ and $\wedge^3 H_{M'}$ is the image
of $\TorelliRel{M'}{\Sigma_g} \subset \TorelliRel{M}{\Sigma_g}$.
\end{lemma}

%%%%%%%%%%%%%%%%%%%%%%%%%%%%%%%%%%%%%%%%%%%%%%%%%%%%%%%%%%%%%%%%%%%%%%%%%%%%%%%%%%%%%%%%%%%%
%%%%%%%%%%%%%%%%%%%%%%%%%%%%%%%%%%%%%%%%%%%%%%%%%%%%%%%%%%%%%%%%%%%%%%%%%%%%%%%%%%%%%%%%%%%%
%%%%%%%%%%%%%%%%%%%%%%%%%%%%%%%%%%%%%%%%%%%%%%%%%%%%%%%%%%%%%%%%%%%%%%%%%%%%%%%%%%%%%%%%%%%%
%%%%%%%%%%%%%%%%%%%%%%%%%%%%%%%%%%%%%%%%%%%%%%%%%%%%%%%%%%%%%%%%%%%%%%%%%%%%%%%%%%%%%%%%%%%%
%%%%%%%%%%%%%%%%%%%%%%%%%%%%%%%%%%%%%%%%%%%%%%%%%%%%%%%%%%%%%%%%%%%%%%%%%%%%%%%%%%%%%%%%%%%%

\section{Generating the Johnson kernel}
\label{section:genjohnsonker}

The main result of this section is as follows.

\begin{theorem}
\label{theorem:genjohnsonker}
Let $S$ be a surface whose genus is at least $2$ and let $S \hookrightarrow \Sigma_g^n$ be a clean
embedding with $n \leq 1$.  Then the group $\JKerRel{S}{\Sigma_g^n}$ is generated by the following
elements.
\begin{compactitem}
\item $\Sigma_g^n$-separating twists.
\item $\Sigma_g^n$-bounding pair maps $T_{\beta} T_{\beta'}^{-1}$, where $\beta$ and $\beta'$
are the boundary components of a component $A$ of $\Sigma_g^n \setminus \Interior(S)$ such that
$A$ is homeomorphic to an annulus and $\partial A \subset S$.
\end{compactitem}
\end{theorem}

The proof of Theorem \ref{theorem:genjohnsonker} is in \S \ref{section:genjohnsonkerproofs} following
a preliminary lemma in \S \ref{section:genjohnsonkerdisc}.  Before doing this, we derive
Theorem \ref{maintheorem:genjohnsonker} from Theorem \ref{theorem:genjohnsonker}.

\begin{proof}[{Proof of Theorem \ref{maintheorem:genjohnsonker}}]
Let us recall the statement.  Let $S$ be a surface whose genus is at least $2$ and let $S \hookrightarrow \Sigma_g^n$
be an embedding with $n \leq 1$.  We want to show that $\JKerRelHat{S}{\Sigma_g^n}$ is generated by
$\{\text{$T_{\gamma}$ $|$ $\gamma$ is a separating curve on $\Sigma_g^n$ and $\gamma \subset S$}\}$.
First, without changing $\JKerRelHat{S}{\Sigma_g^n}$ we can add to $S$ any components of $\Sigma_g^n \setminus \Interior(S)$
that are discs.  Thus we can assume without loss of generality that $S \hookrightarrow \Sigma_g^n$ is clean.
The natural map $\rho : \JKerRel{S}{\Sigma_g^n} \rightarrow \JKerRelHat{S}{\Sigma_g^n}$ is a surjection
which takes $\Sigma_g^n$-separating twists to the claimed generators.  It also satisfies
$\rho(T_{\beta} T_{\beta'}^{-1})=0$ whenever $\beta$ and $\beta'$
are the boundary components of a component $A$ of $\Sigma_g^n \setminus \Interior(S)$ such that
$A$ is homeomorphic to an annulus and $\partial A \subset S$.  The desired result thus follows from Theorem \ref{theorem:genjohnsonker}.
\end{proof}

\subsection{Generating the Johnson kernel disc-pushing subgroup}
\label{section:genjohnsonkerdisc}

In preparation for proving Theorem \ref{theorem:genjohnsonker}, we prove the following.

\begin{lemma}
\label{lemma:discseptwists}
Let $S$ be a surface whose genus is at least $2$, let $S \hookrightarrow \Sigma_g$ be a clean
embedding, and let $\beta$ be a boundary component of $S$.  
Let $\Gamma$ be the subgroup of $\JKerRel{S}{\Sigma_g}$ generated by the purported generating
set from Theorem \ref{theorem:genjohnsonker}.  Then $\PPK{S}{\Sigma_g}{\beta} \subset \Gamma$.
\end{lemma}
\begin{proof}
Let $\hat{S}$ be the result of gluing a disc to $S$ along $\beta$ and let $\pi = \pi_1(\hat{S})$.  
\BeginCases
\begin{case}
$\beta$ separates $\Sigma_g$.
\end{case}
Theorem \ref{theorem:birmanjohnsonsep} says that $\PPK{S}{\Sigma_g}{\beta}$ is the pullback of $[\pi,\pi]_{\Sigma_g}$
under the map 
\[\PPT{S}{\Sigma_g}{\beta} \cong \pi_1(U\hat{S}) \rightarrow \pi_1(\hat{S}).\]
Using Lemma \ref{lemma:relcommgen},
we deduce that $\PPK{S}{\Sigma_g}{\beta}$ is generated by two kinds of elements.
\begin{compactitem}
\item The $\Sigma_g$-separating twist $T_{\beta}$ corresponding to the loop around the fiber in $\pi_1(U\hat{S})$.
\item $\PtPsh{\gamma}$ for $\gamma \in [\pi,\pi]_{\Sigma_g}$ that can be realized by either
a torus-bounding curve or a $\Sigma_g$-boundary curve.
Using the explicit formula in \S \ref{section:modbasic}, we
have $\PtPsh{\gamma} = T_{\tilde{\gamma}_1} T_{\tilde{\gamma}_2}^{-1}$ where $T_{\tilde{\gamma}_1}$
and $T_{\tilde{\gamma}_2}$ are $\Sigma_g$-separating twists.
\end{compactitem}
This implies that $\PPK{S}{\Sigma_g}{\beta} \subset \Gamma$, as desired.

\begin{case}
$\beta$ does not separate $\Sigma_g$ and the component of $\Sigma_g \setminus \Interior(S)$ adjacent to $\beta$ is not
an annulus.
\end{case}
Theorem \ref{theorem:birmanjohnsonnosep} says that $\PPK{S}{\Sigma_g}{\beta}$ is the
subgroup $[\pi,[\pi,\pi]_{\Sigma_g}]$ of $\PPT{S}{\Sigma_g}{\beta} \cong [\pi,\pi]_{\Sigma_g}$.  Define
\[A = \Set{$\gamma \in [\pi,\pi]_{\Sigma_g}$}{$\gamma$ can be realized by a torus-bounding or 
$\Sigma_g$-boundary curve}.\]
Lemma \ref{lemma:relcommgen} says that $A$ is a generating set for $[\pi,\pi]_{\Sigma_g}$.  For $\gamma \in A$, define
\begin{align*}
B_{\gamma} = \{\text{$\delta \in \pi$ $|$ }&\text{$\delta$ can be realized by a nonseparating curve that only}\\
&\text{intersects $\gamma$ at the basepoint}\}.
\end{align*}
For all $\gamma \in A$, the set $B_{\gamma}$ generates $\pi$.  It follows that
the set
\[C = \Set{$[\delta,\gamma]$}{$\gamma \in A$ and $\delta \in B_{\gamma}$}\]
normally generates $[\pi,[\pi,\pi]_{\Sigma_g}]$ as a subgroup of $\pi$.

We now prove that for $\zeta \in \pi$ and $[\delta,\gamma] \in C$, we also have 
$[\zeta \delta \zeta^{-1},\zeta \gamma \zeta^{-1}] \in C$.  By the Dehn-Nielsen-Baer
theorem, inner automorphisms of $\pi$ are induced by homeomorphims of $\hat{S}$ that
fix the basepoint.  It follows that there exists some homeomorphism $f : \hat{S} \rightarrow \hat{S}$
that fixes the basepoint and satisfies $f_{\ast}(x) = \zeta x \zeta^{-1}$ for all $x \in \pi$.  We
conclude that
$[\zeta \delta \zeta^{-1},\zeta \gamma \zeta^{-1}] = [f_{\ast}(\delta),f_{\ast}(\gamma)]$, which 
lies in $C$, as claimed.  This implies that $C$ generates $[\pi,[\pi,\pi]_{\Sigma_g}]$.

\Figure{figure:discjohnsonnosepgen}{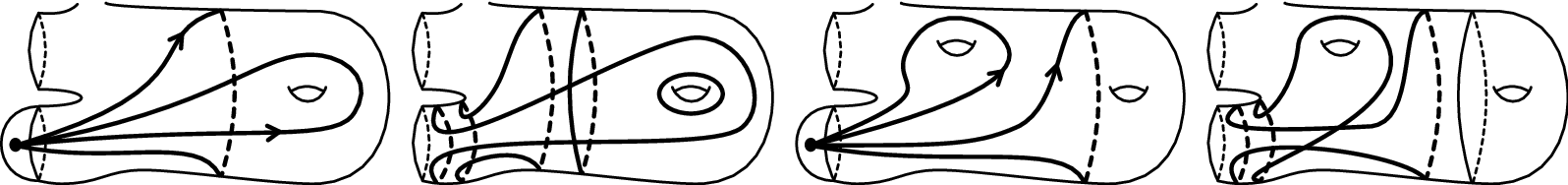}{
a. $\gamma$ is a torus-bounding curve and $\delta$ is inside the torus. \CaptionSpace
b. $\PtPsh{\delta} = T_{\tilde{\delta}_1} T_{\tilde{\delta}_2}^{-1}$ and $\PtPshHat{\gamma} = T_{\tilde{\gamma}_1} (T_{\beta} T_{\tilde{\gamma}_2}^{-1}$ \CaptionSpace
c. $\gamma$ is a torus-bounding curve and $\delta$ is outside the torus. \CaptionSpace
d. $[\PtPsh{\delta},\PtPshHat{\gamma}] = T_{T_{\tilde{\delta}}(\tilde{\gamma})} T_{\tilde{\gamma}}^{-1}$.}

Consider $[\delta,\gamma] \in C$.  The associated element of $\PPK{S}{\Sigma_g}{\beta}$ is $[\PtPsh{\delta},\PtPshHat{\gamma}]$.
We must show that this lies in $\Gamma$.  We will give the details for the case where $\gamma$ is a torus-bounding curve; the case where $\gamma$ is a
$\Sigma_g$-boundary curve is similar.  Let $X$ be the component of $S$ cut along $\gamma$ that is
a one-holed torus.  There are two cases.  In the first,
$\delta$ can be realized by a simple closed nonseparating curve lying inside $X$.  For definiteness,
assume that the orientations on $\gamma$ and $\delta$ are the same as those in Figure
\ref{figure:discjohnsonnosepgen}.a; the other cases are similar.  Letting $\tilde{\gamma}_i$ and $\tilde{\delta}_i$ be the curves
in Figure \ref{figure:discjohnsonnosepgen}.b, the formulas
in \S \ref{section:modbasic} and \S \ref{section:torellibasic} imply that
\[[\PtPsh{\delta},\PtPshHat{\gamma}] = [T_{\tilde{\delta}_1} T_{\tilde{\delta}_2}^{-1}, T_{\tilde{\gamma}_1} (T_{\beta} T_{\tilde{\gamma}_2}^{-1})] = [T_{\tilde{\delta}_1}, T_{\tilde{\gamma}_1}] = T_{\tilde{\delta}_1} T_{\tilde{\gamma}_1} T_{\tilde{\delta}_1}^{-1} T_{\tilde{\gamma}_1}^{-1} = T_{T_{\tilde{\delta}_1}(\tilde{\gamma}_1)} T_{\tilde{\gamma}_1}^{-1}.\]
Both $T_{T_{\tilde{\delta}_1}(\tilde{\gamma}_1)}$ and $T_{\tilde{\gamma}_1}$ are $\Sigma_g$-separating twists, so this lies in $\Gamma$.

The other case is when $\delta$ can be realized by a simple closed nonseparating curve which aside from the basepoint lies entirely
outside $X$.  For definiteness, assume that the orientations on $\gamma$ and $\delta$ are the same as those in Figure \ref{figure:discjohnsonnosepgen}.c;
the other cases are similar.  Letting $\tilde{\gamma}$ and $\tilde{\delta}$ be the curves
in Figure \ref{figure:discjohnsonnosepgen}.d, a similar calculation to the one above shows that
\[[\PtPsh{\delta},\PtPshHat{\gamma}] = T_{T_{\tilde{\delta}}(\tilde{\gamma})} T_{\tilde{\gamma}}^{-1}.\]
Unfortunately, $T_{\tilde{\gamma}}$ is not a $\Sigma_g$-separating twist.  However, let $\epsilon$ be the curve shown
in Figure \ref{figure:discjohnsonnosepgen}.d and let $S'$ be the subsurface of $S$ to the left of $\epsilon$.  Then
we can regard $\phi := T_{T_{\tilde{\delta}}(\tilde{\gamma})} T_{\tilde{\gamma}}^{-1}$ as an element of $\JKerRel{S'}{\Sigma_g}$.  Moreover,
$\phi$ becomes homotopic to the identity when a disc is glued to $S'$ along $\epsilon$, so $\phi \in \PPK{S'}{\Sigma_g}{\epsilon}$.  The
previous case therefore applies to show that $\phi$ can be written as a product of $\Sigma_g$-separating twists, as desired.

\begin{case}
$\beta$ does not separate $\Sigma_g$ and the component of $\Sigma_g \setminus \Interior(S)$ adjacent to $\beta$ is
an annulus.
\end{case}
This is similar to case 2, the only difference being the extra generator for $\PPK{S}{\Sigma_g}{\beta}$ given
by Theorem \ref{theorem:birmanjohnsonnosep}.  As indicated in the remark following Theorem \ref{theorem:birmanjohnsonnosep},
this extra generator lies in $\Gamma$.
\end{proof}

\subsection{The Johnson kernel is generated by separating twists}
\label{section:genjohnsonkerproofs}

We now prove Theorem \ref{theorem:genjohnsonker}.  Our proof will use the following theorem
of the author.

\begin{theorem}
\label{theorem:gentorelli}
Fix some $g \geq 3$.  Let $\alpha$ and $\beta$ be simple closed curves on $\Sigma_g$ that
intersect once.  Then $\Torelli{\Sigma_g}$ is generated by the stabilizer subgroups 
$(\Torelli{\Sigma_g})_{\alpha}$ and $(\Torelli{\Sigma_g})_{\beta}$.
\end{theorem}
\begin{proof}
An immediate consequence of Lemmas 3.1, 5.1, and 5.2 of \cite{PutmanCutPaste}.
\end{proof}

\begin{proof}[{Proof of Theorem \ref{theorem:genjohnsonker}}]
Let us recall the statement.  Let $S$ be a surface whose genus is at least $2$ and let $S \hookrightarrow \Sigma_g^n$ be a clean
embedding with $n \leq 1$.  Then we want to prove that the group $\JKerRel{S}{\Sigma_g^n}$ is generated by the following
elements.
\begin{compactitem}
\item $\Sigma_g^n$-separating twists.
\item $\Sigma_g^n$-bounding pair maps $T_{\beta} T_{\beta'}^{-1}$, where $\beta$ and $\beta'$
are the boundary components of a component $A$ of $\Sigma_g^n \setminus \Interior(S)$ such that
$A$ is homeomorphic to an annulus and $\partial A \subset S$.
\end{compactitem}
If $n=1$, then there is a clean embedding $\Sigma_g^n \hookrightarrow \Sigma_{g+1}$ and Lemma \ref{lemma:noambiguity}
says that $\JKerRel{S}{\Sigma_g^n} = \JKerRel{S}{\Sigma_{g+1}}$.  We can therefore assume without loss of generality
that $n=0$.

Let $\Gamma(S,\Sigma_g)$ be the subgroup of $\JKerRel{S}{\Sigma_g}$ generated by the above elements.
If $S = \Sigma_g$, then we will write $\Gamma(\Sigma_g)$ instead of $\Gamma(S,\Sigma_g)$.
We want to prove that $\Gamma(S,\Sigma_g) = \JKerRel{S}{\Sigma_g}$.  We break this into several steps.

\BeginStepsb
\begin{stepb}
\label{step:genus2}
$\Gamma(\Sigma_2) = \JKer{\Sigma_2}$.
\end{stepb}

The Johnson homomorphism $\tau_{\Sigma_2}$ vanishes, so $\JKer{\Sigma_2} = \Torelli{\Sigma_2}$ and
the result follows from the fact that $\Torelli{\Sigma_2}$ is generated by separating twists (this
theorem should probably be attributed to Powell \cite{PowellTorelli}, though he does not state it;
see \cite{PutmanCutPaste} and \cite{HatcherMargalit} for modern proofs).

\begin{stepb}
\label{step:boundaryreduction}
For some $h \geq 2$, assume that $\Gamma(\Sigma_h) = \JKer{\Sigma_h}$.  
Then $\Gamma(S,\Sigma_g) = \JKerRel{S}{\Sigma_g}$
whenever $S$ has genus $h$.
\end{stepb}

The proof is by induction on the number of boundary components of $S$.  For the base case when $S$ is a closed surface (and thus $g=h$ and $S = \Sigma_g$),
the result holds by assumption.  Now assume that $S$ has nonempty boundary.  There are three cases.  
The first is that $S$ has one boundary component.  Letting $\beta$ be the boundary component of $S$ and
$\hat{S}$ be the result of gluing a disc to $\beta$,
Theorem \ref{theorem:birmanjohnsonone} says that there is a short exact sequence
\[1 \longrightarrow \PPK{S}{\Sigma_g}{\beta} \longrightarrow \JKerRel{S}{\Sigma_g} \longrightarrow \JKer{\hat{S}} \longrightarrow 1.\]
Lemma \ref{lemma:discseptwists} says that $\PPK{S}{\Sigma_g}{\beta}$ 
is contained in $\Gamma(S,\Sigma_g)$.  Also, examining our purported generating set we see that $\Gamma(S,\Sigma_g)$
projects to $\Gamma(\hat{S})$.  We have assumed that $\Gamma(\Sigma_h) = \JKer{\Sigma_h}$, so
we conclude that $\Gamma(S,\Sigma_g) = \JKerRel{S}{\Sigma_g}$.

The second case is when $S$ has multiple boundary components, but all of them separate $\Sigma_g$.
Let $\beta$ be one of the boundary components of $S$ and $S' \hookrightarrow S$ be a splitting surface
for $\beta$.  Theorem \ref{theorem:birmanjohnsonsep} says that
\[\JKerRel{S}{\Sigma_g} = \PPK{S}{\Sigma_g}{\beta} \rtimes \JKerRel{S'}{\Sigma_g}.\]
Lemma \ref{lemma:discseptwists} says that we have $\PPK{S}{\Sigma_g}{\beta} \subset \Gamma(S,\Sigma_g)$. 
Our inductive
hypothesis says that $\JKerRel{S'}{\Sigma_g} = \Gamma(S',\Sigma_g)$; examining our purported generating set,
we see that every purported generator for $\JKerRel{S'}{\Sigma_g}$ is also one for the larger
group $\JKerRel{S}{\Sigma_g}$, so we deduce that $\JKerRel{S'}{\Sigma_g} \subset \Gamma(S,\Sigma_g)$.  We conclude
that $\JKerRel{S}{\Sigma_g} = \Gamma(S,\Sigma_g)$, as desired.

Finally, if $S$ has multiple boundary components and one of them does not separate $\Sigma_g$, then we
can substitute Theorem \ref{theorem:birmanjohnsonnosep} for Theorem \ref{theorem:birmanjohnsonsep} (and
use a $\Sigma_g$-splitting surface) in the above argument.

\begin{stepb}
\label{step:genusreduction}
For some $g \geq 3$, assume that $\Gamma(S,\Sigma_g) = \JKerRel{S}{\Sigma_g}$ for all
cleanly embedded subsurfaces $S$ of $\Sigma_g$ whose genus is less than $g$.  Then
$\Gamma(\Sigma_g)=\JKer{\Sigma_g}$.
\end{stepb}

Set $H = \HH_1(\Sigma_g;\Z)$ and $\mathcal{Q} = \Torelli{\Sigma_g} / \Gamma(\Sigma_g)$.
The surjection $\tau_{\Sigma_g} : \Torelli{\Sigma_g} \rightarrow (\wedge^3 H)/H$ induces
a surjection $\psi : \mathcal{Q} \rightarrow (\wedge^3 H)/H$.  Our goal is to show that
$\psi$ is an isomorphism.  We divide this into four substeps.  Let $\alpha$ and $\beta$
be two curves on $\Sigma_g$ that intersect once.

\BeginSubStepsb
\begin{substepb}
\label{substep:makephi}
Let $x$ be either $\alpha$ or $\beta$.  We construct groups $\mathcal{Q}_x$ and maps $\phi_x : \mathcal{Q}_x \rightarrow \mathcal{Q}$.
\end{substepb}
Let $\mathcal{A}_x$ be an annulus in $\Sigma_g$ one of whose boundary components is $x$
and let $M_x = \Sigma_g \setminus \Interior(\mathcal{A}_x)$.  Thus $M_x \cong \Sigma_{g-1}^2$ and we have
a surjection $\TorelliRel{M_x}{\Sigma_g} \rightarrow (\Torelli{\Sigma_g})_x$.  Our assumptions
say that $\Gamma(M_x,\Sigma_g) = \JKerRel{M_x}{\Sigma_g}$, and image of $\Gamma(M_x,\Sigma_g)$ under
the map $\TorelliRel{M_x}{\Sigma_g} \rightarrow \Torelli{\Sigma_g}$ is contained in
$\Gamma(\Sigma_g)$.  Defining $\mathcal{Q}_x = \TorelliRel{M_x}{\Sigma_g} / \JKerRel{M_x}{\Sigma_g}$, we
get an induced map $\phi_x : \mathcal{Q}_x \rightarrow \mathcal{Q}$.

By definition,
\[\mathcal{Q}_x \cong \tau_{M_x,\Sigma_g}(\TorelliRel{M_x}{\Sigma_g}) \subset (\wedge^3 H)/H.\]
Let $M' \subset \Sigma_g$ be the complement of an open regular neighborhood of $\alpha \cup \beta$.  
Thus $M' \subset M_x$ and $M' \cong \Sigma_{g-1}^1$.  Moreover, 
$M'$ is a $\Sigma_g$-splitting surface for the boundary
component $x$ of $M_x$.  Letting $H_{M'} = \HH_1(M';\Z) \subset H$, Lemma \ref{lemma:nosepdiscpushcalc}
implies that
\[\mathcal{Q}_x \cong \left((\wedge^2 H_{M'})/\Z\right) \oplus \left(\wedge^3 H_{M'}\right) \subset (\wedge^3 H)/H.\]
We will identify $\mathcal{Q}_x$ with $\left((\wedge^2 H_{M'})/\Z\right) \oplus \left(\wedge^3 H_{M'}\right)$.  
The next two substeps focus on the effect of $\phi_x$ on these two summands.

\begin{substepb}
\label{substep:sameimage}
For all $v \in \wedge^3 H_{M'}$, we have $\phi_{\alpha}(v) = \phi_{\beta}(v)$.
\end{substepb}
Since $\wedge^3 H_{M'} \subset \mathcal{Q}_{\alpha}$ and $\wedge^3 H_{M'} \subset \mathcal{Q}_{\beta}$ are the 
images of $\TorelliRel{M'}{\Sigma_g}$, this follows from the commutative diagram
\[\xymatrix{
& \TorelliRel{M_{\alpha}}{\Sigma_g} \ar[rd] & \\
\TorelliRel{M'}{\Sigma_g} \ar[ru] \ar[rd] & & \Torelli{\Sigma_g} \\
& \TorelliRel{M_{\beta}}{\Sigma_g} \ar[ru]  & }\]

\begin{substepb}
\label{substep:commute}
The images $\phi_{\alpha}((\wedge^2 H_{M'})/\Z) \subset \mathcal{Q}$ and 
$\phi_{\beta}((\wedge^2 H_{M'})/\Z) \subset \mathcal{Q}$ commute.
\end{substepb}

For $x$ equal to either $\alpha$ or $\beta$, let $\hat{M}_x$ be the result
of gluing a disc to $M_x$ along $x$ (see Figures \ref{figure:commute}a--c).  
Recall from \S \ref{section:torellibasic} that
$\PPT{M_x}{\Sigma_g}{x} \cong [\pi_1(\hat{M}_x),\pi_1(\hat{M}_x)]$ and that for
$\gamma \in [\pi_1(\hat{M}_x),\pi_1(\hat{M}_x)]$ the associated element
of $\PPT{M_x}{\Sigma_g}{x}$ is denoted $\PtPshHat{\gamma}$.  Let
$\rho_x : \TorelliRel{M_x}{\Sigma_g} \rightarrow \mathcal{Q}_x$ be the projection.

Lemma \ref{lemma:commgen} says that $[\pi_1(\hat{M}_{\alpha}),\pi_1(\hat{M}_{\alpha})]$
is generated by torus-bounding curves $a$.  Fixing such an 
$a \in [\pi_1(\hat{M}_{\alpha}),\pi_1(\hat{M}_{\alpha})]$, it is enough
to find a set $B \subset [\pi_1(\hat{M}_{\beta}),\pi_1(\hat{M}_{\beta})]$ with the following
two properties.
\begin{compactenum}
\item The set $\Set{$\rho_{\beta}(\PtPshHat{b})$}{$b \in B$}$ generates 
$(\wedge^2 H_{M'})/\Z \subset \mathcal{Q}_{\beta}$.
\item For all $b \in B$, we have $[\phi_{\alpha}(\rho_{\alpha}(\PtPshHat{a})),\phi_{\beta}(\rho_{\beta}(\PtPshHat{b}))]=1 \in \mathcal{Q}$.
\end{compactenum}

\Figure{figure:commute}{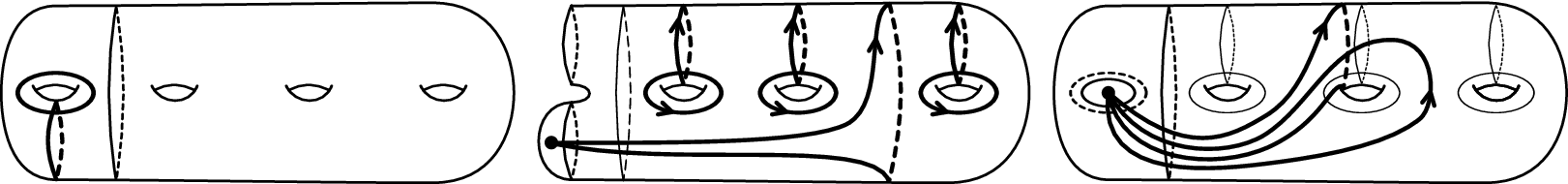}{a. Two curves $\alpha$ and $\beta$ that intersect once and
the complement $M'$ of an open regular neighborhood of $\alpha \cup \beta$. \CaptionSpace
b. The surface $\hat{M}_{\alpha}$, the torus-bounding curve $a \in \pi_1(\hat{M}_{\alpha})$, and
the curves $e_1,\ldots,e_{2g-2}$.  \CaptionSpace
c. The surface $\hat{M}_{\beta}$; the dotted loop on the left hand side is a boundary component
on the back of the surface.  As an example of how to find the curves $b_i$, the curves
$b_3$ and $b_4$ are shown.  The curves $e_1,\ldots,e_{2g-2}$
are drawn faintly.}

As shown in Figure \ref{figure:commute}.b, the curve $a$ divides $M'$ into two components $X$ and $Y$
with $Y$ a one-holed torus.  Choose unbased oriented simple closed curves $e_1,\ldots,e_{2g-2}$ in $M'$
with the following properties (see Figure \ref{figure:commute}.b).
\begin{compactitem}
\item $e_1,\ldots,e_{2g-4} \subset X$ and $e_{2g-3},e_{2g-2} \subset Y$.
\item For $1 \leq i < j \leq 2g-2$, the curves $e_i$ and $e_j$ are disjoint unless $i=2k$ and
$j=2k+1$ for some $1 \leq k \leq g-1$, in which case they intersect once with positive sign.
\end{compactitem}
This implies that $[e_1],\ldots,[e_{2g-2}]$ is a symplectic basis for $H_{M'}$, and thus
that $\wedge^2 H_{M'}$ is generated by $\Set{$[e_i] \wedge [e_j]$}{$1 \leq i < j \leq 2g-2$}$.

As is shown in Figure \ref{figure:commute}.c, we can find elements 
$b_1,\ldots,b_{2g-2} \in \pi_1(\hat{M}_{\beta})$ with the following properties.
\begin{compactitem}
\item For $1 \leq i \leq 2g-2$, the curve $b_i$ is freely homotopic to and disjoint from $e_i$.
\item For $1 \leq i,j \leq 2g-2$ such that $i \neq j$ and such that $e_i$ is disjoint from $e_j$, 
the curve $b_i$ is also disjoint from $e_j$.
\end{compactitem}
The first of these shows that for $1 \leq i < j \leq 2g-2$, the element
$\rho_{\beta}(\PtPshHat{[b_i,b_j]}) \in (\wedge^2 H_{M'})/\Z \subset \mathcal{Q}_{\beta}$ equals
the image of $[e_i] \wedge [e_j] \in \wedge^2 H_{M'}$ in $(\wedge^2 H_{M'})/\Z$.  It follows
that the set
\[B = \Set{$[b_i,b_j]$}{$1 \leq i < j \leq 2g-2$}\]
satisfies the first property of $B$ listed above.

It remains to prove the second claimed property of $B$.  Fixing $1 \leq i < j \leq 2g-2$,
this asserts that
\begin{equation}
\label{eqn:toverify}
[\phi_{\alpha}(\rho_{\alpha}(\PtPshHat{a})),\phi_{\beta}(\rho_{\beta}(\PtPshHat{[b_i,b_j]}))]=1.
\end{equation}
Let $i_{\alpha} : \TorelliRel{M_{\alpha}}{\Sigma_g} \rightarrow \Torelli{\Sigma_g}$ and
$i_{\beta} : \TorelliRel{M_{\beta}}{\Sigma_g} \rightarrow \Torelli{\Sigma_g}$ be the maps
induced by the inclusions $M_{\alpha} \hookrightarrow \Sigma_g$ and $M_{\beta} \hookrightarrow \Sigma_g$.
The equation \eqref{eqn:toverify} is equivalent to the assertion that the element
$\theta:=[i_{\alpha}(\PtPshHat{a}),i_{\beta}(\PtPshHat{[b_i,b_j]}] \in \JKer{\Sigma_g}$ lies
in $\Gamma(\Sigma_g)$.

By construction (and the fact that $g \geq 3$), 
there is some $1 \leq k \leq 2g-2$ such that $e_k$ is disjoint from
both $b_i$ and $b_j$.  Since $e_k$ is also disjoint from $a$, we see that
$e_k$ is fixed by $\PtPshHat{[b_i,b_j]}$ and $\PtPshHat{a}$, and thus also fixed
by $\theta$.  Letting $R$ be the complement in
$\Sigma_g$ of an open regular neighborhood of $e_k$, we deduce that $\theta$
lies in the image of $\JKerRel{R}{\Sigma_g}$.  By assumption, we have that
$\Gamma(R,\Sigma_g) = \JKerRel{R}{\Sigma_g}$, so we conclude that $\theta \in \Gamma(\Sigma_g)$,
as desired.

\begin{substepb}
\label{substep:finishoff}
We use $\phi_{\alpha}$ and $\phi_{\beta}$ to prove that $\psi$ is an isomorphism.
\end{substepb}
The universal property of the free product says that $\phi_{\alpha}$ and $\phi_{\beta}$ can be combined 
into a map
\[\phi_{\alpha} \ast \phi_{\beta} : \mathcal{Q}_{\alpha} \ast \mathcal{Q}_{\beta} \cong \left[\left((\wedge^2 H_{M'})/\Z\right) \oplus \left(\wedge^3 H_{M'}\right)\right] \ast \left[\left((\wedge^2 H_{M'})/\Z\right) \oplus \left(\wedge^3 H_{M'}\right)\right] \longrightarrow \mathcal{Q}.\]
Theorem \ref{theorem:gentorelli} implies that $\phi_{\alpha} \ast \phi_{\beta}$ is surjective, and Substeps 
\ref{substep:sameimage} and \ref{substep:commute} show that $\phi_{\alpha} \ast \phi_{\beta}$ factors
through a surjective homomorphism
\[\phi : \left[(\wedge^2 H_{M'})/\Z\right] \oplus \left[(\wedge^2 H_{M'})/\Z\right] \oplus \left[\wedge^3 H_{M'}\right] \longrightarrow \mathcal{Q}.\]

Summing up, we have surjective maps
of the form
\[\left[(\wedge^2 H_{M'})/\Z\right] \oplus \left[(\wedge^2 H_{M'})/\Z\right] \oplus \left[\wedge^3 H_{M'}\right] \stackrel{\phi}{\twoheadlongrightarrow} \mathcal{Q} \stackrel{\psi}{\twoheadlongrightarrow} (\wedge^3 H)/H.\]
Since $g \geq 3$ (and, in particular, $g > 1$), the group 
$(\wedge^3 H)/H$ is a free abelian group of rank $\binom{2g}{3} - 2g$.  Also, 
$\left[(\wedge^2 H_{M'})/\Z\right] \oplus \left[(\wedge^2 H_{M'})/\Z\right] \oplus \left[\wedge^3 H_{M'}\right]$
is a free abelian group of rank
\[\left[\binom{2g-2}{2}-1\right] + \left[\binom{2g-2}{2}-1\right] + \binom{2g-2}{3}.\]
A little elementary algebra shows that this also equals $\binom{2g}{3} - 2g$.  In other words, $\psi \circ \phi$ is a surjective map
between free abelian groups of the same rank, so it must be an isomorphism.  We conclude that $\psi$ is also an isomorphism, as desired.
\end{proof}

%%%%%%%%%%%%%%%%%%%%%%%%%%%%%%%%%%%%%%%%%%%%%%%%%%%%%%%%%%%%%%%%%%%%%%%%%%%%%%%%%%%%%%%%%%%%
%%%%%%%%%%%%%%%%%%%%%%%%%%%%%%%%%%%%%%%%%%%%%%%%%%%%%%%%%%%%%%%%%%%%%%%%%%%%%%%%%%%%%%%%%%%%
%%%%%%%%%%%%%%%%%%%%%%%%%%%%%%%%%%%%%%%%%%%%%%%%%%%%%%%%%%%%%%%%%%%%%%%%%%%%%%%%%%%%%%%%%%%%
%%%%%%%%%%%%%%%%%%%%%%%%%%%%%%%%%%%%%%%%%%%%%%%%%%%%%%%%%%%%%%%%%%%%%%%%%%%%%%%%%%%%%%%%%%%%
%%%%%%%%%%%%%%%%%%%%%%%%%%%%%%%%%%%%%%%%%%%%%%%%%%%%%%%%%%%%%%%%%%%%%%%%%%%%%%%%%%%%%%%%%%%%

\section{Finite-index subgroups of the Torelli group}
\label{section:abeltorelli}

In this section, we prove Theorem \ref{maintheorem:abeltorelli}.  We first recall its statement.
Let $S$ be a surface whose genus is at least $3$ and $S \hookrightarrow \Sigma_g^n$ an embedding
with $n \leq 1$.  Also, let $\Gamma$ be a finite-index subgroup of $\TorelliRelHat{S}{\Sigma_g^n}$
with $\JKerRelHat{S}{\Sigma_g^n} < \Gamma$.  Then Theorem \ref{maintheorem:abeltorelli} asserts that
$\HH_1(\Gamma;\Q) \cong \tau_{\Sigma_g^n}(\TorelliRelHat{S}{\Sigma_g^n}) \otimes \Q$.

The target of $\tau_{\Sigma_g^n}$ is a free abelian group and $[\TorelliRelHat{S}{\Sigma_g^n}:\Gamma] < \infty$, so
\[\tau_{\Sigma_g^n}(\TorelliRelHat{S}{\Sigma_g^n}) \otimes \Q \cong \tau_{\Sigma_g^n}(\Gamma) \otimes \Q.\]
Since $\JKerRelHat{S}{\Sigma_g^n} < \Gamma$, we have $\Ker(\tau_{\Sigma_g^n}|_{\Gamma}) = \JKerRelHat{S}{\Sigma_g^n}$, so
Theorem \ref{maintheorem:genjohnsonker} implies that $\Ker(\tau_{\Sigma_g^n}|_{\Gamma})$ is generated by separating twists.
We conclude that Theorem \ref{maintheorem:abeltorelli} is equivalent to the following lemma.  If $G$ is a group and $g \in G$, then
denote by $[g]_{G}$ the associated element of $\HH_1(G;\Q)$.

\begin{lemma}[Separating twists vanish]
\label{lemma:septwistsvanish}
For $g \geq 3$ and $n \leq 1$, let $S \hookrightarrow \Sigma_g^n$ be
a compact connected subsurface of genus at least $3$.
Let $\Gamma$ be a finite-index subgroup of $\TorelliRelHat{S}{\Sigma_g^n}$
with $\JKerRelHat{S}{\Sigma_g^n} < \Gamma$ and let $\delta$ be a simple closed
curve on $S$ which separates $\Sigma_g^n$.  Then $[T_{\delta}]_{\Gamma} = 0$.
\end{lemma}

\begin{remark}
In the special case that $S = \Sigma_g^n$ and $\Gamma = \TorelliRelHat{S}{\Sigma_g^n}$, this was proven
by Johnson in \cite{JohnsonAbel}.  Our proof builds on Johnson's proof.
\end{remark}

The proof of Lemma \ref{lemma:septwistsvanish} is contained in \S \ref{section:vanishing2}.  This
is proceeded by \S \ref{section:vanishing1}.

\subsection{Separating twists are torsion, special case}
\label{section:vanishing1}

In this section, we prove the following lemma.

\begin{lemma}
\label{lemma:septwistsvanishweak}
For $g \geq 2$ and $n \leq 1$, let $S \hookrightarrow \Sigma_g^n$ be
a compact connected subsurface such that $S \cong \Sigma_h^2$ with $h \geq 2$ and such that
both boundary components of $S$ separate $\Sigma_g^n$.  Let $\beta$ be a boundary
component of $S$.  Then $[T_{\beta}]_{\TorelliRelHat{S}{\Sigma_g^n}} = 0$.
\end{lemma}

The proof requires the following lemma.  Recall that we defined bounding pair maps in \S \ref{section:torellibasic}.  Also,
if $\gamma$ is an oriented simple closed curve on $S$, write $[\gamma]$ for its homology class in $\HH_1(\Sigma_g^n;\Z)$.

\begin{lemma}
\label{lemma:conjugacyclassbp}
For $g \geq 2$ and $n \leq 1$, let $S \hookrightarrow \Sigma_g^n$ be
a compact connected subsurface such that $S \cong \Sigma_h^2$ and such that
both boundary components of $S$ separate $\Sigma_g^n$.  Let $\beta$ be a boundary
component of $S$ and let $T_{x_1} T_{x_2}^{-1}$ and $T_{x_1'} T_{x_2'}^{-1}$
be $\Sigma_g^n$-bounding pair maps in $\TorelliRel{S}{\Sigma_g^n}$ that satisfy the following conditions.
\begin{compactenum}
\item For $1 \leq i \leq 2$, the $1$-submanifolds $\beta \cup x_1 \cup x_2$  and $\beta \cup x_1' \cup x_2'$ 
of $S$ each bound subsurfaces homeomorphic to $\Sigma_{0,3}$ (see Figure \ref{figure:septwistsvanish}.a).
\item Orient $x_1$ and $x_2$ (resp.\ $x_1'$ and $x_2'$) so that $\beta$ is in the component of $S$ cut along $x_1 \cup x_2$ (resp.\ $x_1' \cup x_2'$) to the left
of $x_1$ and $x_2$ (resp.\ $x_1'$ and $x_2'$).  Then $[x_1] = [x_1']$.
\end{compactenum}
Then $T_{x_1} T_{x_2}^{-1}$ and $T_{x_1'} T_{x_2'}^{-1}$ are conjugate elements of $\TorelliRelHat{S}{\Sigma_g^n}$.
\end{lemma}

This is similar to a result of Johnson \cite[Theorem 1B]{JohnsonConj} and can be proved in exactly the same way.

\begin{remark}
The statement of \cite[Theorem 1B]{JohnsonConj} contains an extra condition about a certain splitting of homology, but
that condition is trivially satisfied in our situation.
\end{remark}

\Figure{figure:septwistsvanish}{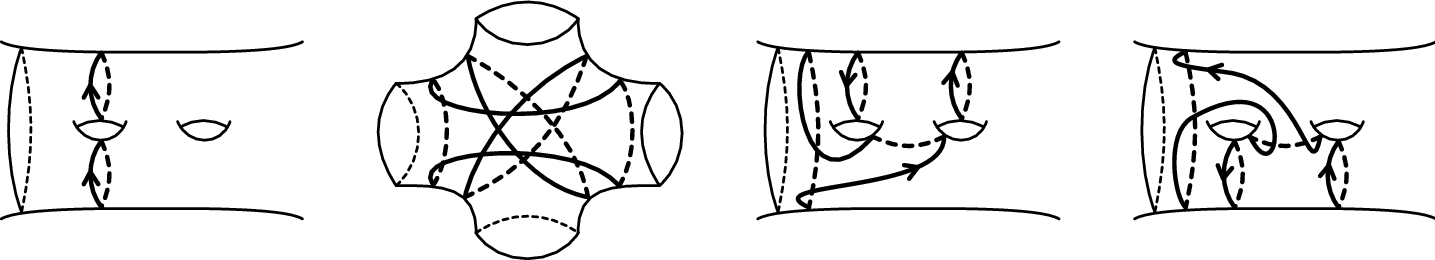}
{a. Curves like in Lemma \ref{lemma:conjugacyclassbp}
\CaptionSpace b. A lantern relation $T_{\beta} = (T_{x_1} T_{x_2}^{-1}) (T_{y_1} T_{y_2}^{-1}) (T_{z_1} T_{z_2}^{-1})$
\CaptionSpace c,d. Two lantern relations as indicated in the proof of Lemma \ref{lemma:septwistsvanishweak}.  We
only draw the four ``exterior'' curves, namely $\beta \cup x_2 \cup y_2 \cup z_2$ and $\beta \cup x_2' \cup y_2' \cup z_2'$.}

\begin{proof}[{Proof of Lemma \ref{lemma:septwistsvanishweak}}]
We will use the {\em lantern relation}, which is the relation
\[T_{\beta} = (T_{x_1} T_{x_2}^{-1}) (T_{y_1} T_{y_2}^{-1}) (T_{z_1} T_{z_2}^{-1}),\]
where the curves $x_i$ and $y_i$ and $\beta$ are the curves on the $4$-holed sphere $\Sigma_0^4$ shown 
in Figure \ref{figure:septwistsvanish}.b (see \cite[\S 7g]{Dehn}).
Let $x_2, y_2, z_2 \subset S$ be the curves shown in Figure \ref{figure:septwistsvanish}.c.  The
multicurve $\beta \cup x_2 \cup y_2 \cup z_2$ bounds a subsurface $\mathcal{L}$ of $S$ with $\mathcal{L} \cong \Sigma_0^4$.
Inside this subsurface $\mathcal{L}$, we can find curves $x_1 \cup y_1 \cup z_1$ such that
we have a lantern relation
\begin{equation}
\label{eqn:lantern1}
T_{\beta} = (T_{x_1} T_{x_2}^{-1}) (T_{y_1} T_{y_2}^{-1}) (T_{z_1} T_{z_2}^{-1}).
\end{equation}
We remark that the curves $x_1 \cup y_1 \cup z_1$ are not unique (they depend on a choice of homeomorphism
$\mathcal{L} \cong \Sigma_0^4$), but any choice will work in our proof.
Similarly, if $x_2', y_2', z_2' \subset S$ are the curves in Figure \ref{figure:septwistsvanish}.d, then
we can find curves $x_1' \cup y_2' \cup z_2'$ such that we have a lantern relation 
\begin{equation}
\label{eqn:lantern2}
T_{\beta} = (T_{x_1'} T_{x_2'}^{-1}) (T_{y_1'} T_{y_2'}^{-1}) (T_{z_1'} T_{z_2'}^{-1}).
\end{equation}
Orient $x_1$ and $x_2$ (resp.\ $x_1'$ and $x_2'$) so that $\beta$ is in the component of $S$ cut along $x_1 \cup x_2$ (resp.\ $x_1' \cup x_2'$) to the left
of $x_1$ and $x_2$ (resp.\ $x_1'$ and $x_2'$).  This choice of orientation yields $[x_1] = -[x_2]$ and
$[x_1'] = -[x_2']$.  As shown in Figure \ref{figure:septwistsvanish}.c,d, we have
$[x_2] = -[x_2']$, so
\[[x_1] = -[x_2] = [x_2'].\]
Lemma \ref{lemma:conjugacyclassbp} thus implies
that $T_{x_1} T_{x_2}^{-1}$ is conjugate in $\TorelliRelHat{S}{\Sigma_g^n}$ to $T_{x_2'} T_{x_1'}^{-1}$
(notice that we are using $T_{x_2'} T_{x_1'}^{-1}$ here and not $T_{x_1'} T_{x_2'}^{-1}$),
and hence
$[T_{x_1} T_{x_2}^{-1}]_{\TorelliRelHat{S}{\Sigma_g^n}} = -[T_{x_1'} T_{x_2'}^{-1}]_{\TorelliRelHat{S}{\Sigma_g^n}}$.  Similarly, we have
$[T_{y_1} T_{y_2}^{-1}]_{\TorelliRelHat{S}{\Sigma_g^n}} = -[T_{y_1'} T_{y_2'}^{-1}]_{\TorelliRelHat{S}{\Sigma_g^n}}$ and
$[T_{z_1} T_{z_2}^{-1}]_{\TorelliRelHat{S}{\Sigma_g^n}} = -[T_{z_1'} T_{z_2'}^{-1}]_{\TorelliRelHat{S}{\Sigma_g^n}}$.
Combining these facts with \eqref{eqn:lantern1} and \eqref{eqn:lantern2}, we deduce that
\begin{align*}
2[T_{\beta}]_{\TorelliRelHat{S}{\Sigma_g^n}} = (&[T_{x_1} T_{x_2}^{-1}]_{\TorelliRelHat{S}{\Sigma_g^n}} + [T_{y_1} T_{y_2}^{-1}]_{\TorelliRelHat{S}{\Sigma_g^n}} + [T_{z_1} T_{z_2}^{-1}]_{\TorelliRelHat{S}{\Sigma_g^n}})\\ 
&+ ([T_{x_1'} T_{x_2'}^{-1}]_{\TorelliRelHat{S}{\Sigma_g^n}} + [T_{y_1'} T_{y_2'}^{-1}]_{\TorelliRelHat{S}{\Sigma_g^n}} + [T_{z_1'} T_{z_2'}^{-1}]_{\TorelliRelHat{S}{\Sigma_g^n}}) = 0,
\end{align*}
as desired.
\end{proof}

\subsection{Separating twists are torsion, general case}
\label{section:vanishing2}

In this section, we derive Lemma \ref{lemma:septwistsvanish} from Lemma \ref{lemma:septwistsvanishweak}.  This requires
the following lemma, which is an abstraction of the main idea of the proof of
\cite[Theorem 1.1]{PutmanFiniteIndexNote}.

\begin{lemma}
\label{lemma:finiteindexvanish}
Let $G$ be a group, let $G' < G$ be a finite index subgroup, and let $g \in G$ be
a central element.  Assume that $[g]_G = 0$ and that $g \in G'$.  Then $[g]_{G'}=0$.
\end{lemma}
\begin{proof}
We can assume that $g$ has infinite order, so $\Span{g} \cong \Z$.  Since $g$ is central, we have $\Span{g} \lhd G$.
Define $\overline{G} = G / \Span{g}$ and $\overline{G}' = G' / \Span{g}$.  We have
a commutative diagram of central extensions
\[\begin{CD}
1 @>>> \Z          @>>> G'   @>>> \overline{G}' @>>> 1  \\
@.     @VV{\cong}V      @VVV      @VVV               @. \\
1 @>>> \Z          @>>> G    @>>> \overline{G}  @>>> 1
\end{CD}\]
There is an associated map of 5-term exact sequences in rational
group homology (see \cite[Corollary VII.6.4]{BrownCohomology}), which takes the form
\[\begin{CD}
\HH_2(G';\Q) @>>> \HH_2(\overline{G}';\Q) @>{j'}>> \Q          @>{i'}>> \HH_1(G';\Q) @>>> \HH_1(\overline{G}';\Q) @>>> 0  \\
@VVV              @VV{\phi}V                       @VV{\cong}V          @VVV              @VVV                         @. \\
\HH_2(G;\Q)  @>>> \HH_2(\overline{G};\Q)  @>{j}>>  \Q          @>{i}>>  \HH_1(G;\Q)  @>>> \HH_1(\overline{G};\Q)  @>>> 0  \\
\end{CD}\]
By assumption, the map $i$ in this diagram is the $0$ map, so $j$ is surjective.  Since $\overline{G}'$ is a finite
index subgroup of $\overline{G}$, we can make use of the transfer
map in group homology (see \cite[Chapter III.9]{BrownCohomology}) to deduce that $\phi$ is surjective.  We
conclude that $j'$ is surjective, so $i' = 0$, and the lemma follows.
\end{proof}

\begin{proof}[{Proof of Lemma \ref{lemma:septwistsvanish}}]
Since $S$ has genus at least $3$, we can find a subsurface $S'$ of $S$ such that $S' \cong \Sigma_2^2$ and
such that $\delta$ is a boundary component of $S'$.  Since $\delta$ separates $\Sigma_g^n$, the other boundary component
of $S'$ also separates $\Sigma_g^n$.  Setting
$\Gamma' = \Gamma \cap \TorelliRelHat{S'}{\Sigma_g^n}$, it is enough to show that $[T_{\delta}]_{\Gamma'} = 0$.
Lemma \ref{lemma:septwistsvanishweak} implies that $[T_{\delta}]_{\TorelliRelHat{S'}{\Sigma_g^n}} = 0$.
Since $T_{\delta}$ is central in $\TorelliRelHat{S'}{\Sigma_g^n}$ and $\Gamma'$ is a finite-index subgroup of $\TorelliRelHat{S'}{\Sigma_g^n}$,
Lemma \ref{lemma:finiteindexvanish} implies that $[T_{\delta}]_{\TorelliRel{\Sigma}{S'}} = 0$, as desired.
\end{proof}

%%%%%%%%%%%%%%%%%%%%%%%%%%%%%%%%%%%%%%%%%%%%%%%%%%%%%%%%%%%%%%%%%%%%%%%%%%%%%%%%%%%%%%%%%%%%
%%%%%%%%%%%%%%%%%%%%%%%%%%%%%%%%%%%%%%%%%%%%%%%%%%%%%%%%%%%%%%%%%%%%%%%%%%%%%%%%%%%%%%%%%%%%
%%%%%%%%%%%%%%%%%%%%%%%%%%%%%%%%%%%%%%%%%%%%%%%%%%%%%%%%%%%%%%%%%%%%%%%%%%%%%%%%%%%%%%%%%%%%
%%%%%%%%%%%%%%%%%%%%%%%%%%%%%%%%%%%%%%%%%%%%%%%%%%%%%%%%%%%%%%%%%%%%%%%%%%%%%%%%%%%%%%%%%%%%
%%%%%%%%%%%%%%%%%%%%%%%%%%%%%%%%%%%%%%%%%%%%%%%%%%%%%%%%%%%%%%%%%%%%%%%%%%%%%%%%%%%%%%%%%%%%

\noindent
{\raggedright
Andrew Putman\\
Department of Mathematics\\
Rice University, MS 136 \\
6100 Main St.\\
Houston, TX 77005\\
E-mail: {\tt andyp@rice.edu}}

\end{document}